\documentclass[12pt,reqno]{amsart}
\usepackage[left=3cm,top=2cm,right=3cm,bottom=2cm]{geometry}
\usepackage{amssymb}
\usepackage{amsbsy}
\usepackage{epsfig}
\usepackage{tikz}
\usepackage[english]{babel}    
\usepackage{enumitem}
\usepackage{multicol}
\usepackage{blkarray}
\usepackage{tikz}
\usepackage{mathtools}
\usepackage{amsthm}
\usepackage{float}
\usepackage{tikz-network}

\usetikzlibrary{decorations.pathreplacing}

\usepackage{nicematrix}

\usepackage{algorithm}
\usepackage{algpseudocode}

\usepackage{graphicx,subcaption} 

\newtheorem{teore}{Theorem}[section] 
\newtheorem{theorem}[teore]{Theorem} 

\newtheorem{defn}[teore]{Definition}
\newtheorem{lemat}[teore]{Lemma}
\newtheorem{coro}[teore]{Corollary}

\theoremstyle{definition}

\newtheorem{claim}{Claim}

\DeclareMathOperator*{\Spec}{Spec}
\DeclareMathOperator*{\DSpec}{DSpec}

\DeclareMathOperator*{\tr}{tr}

\newcommand{\um}{
\begin{tikzpicture}[scale=.7,auto=left,every node/.style={circle,scale=0.5}]

\path ( 0,-0.7)node[shape=circle,draw,fill=black] (10) {}
      (1,-0.7)node[shape=circle,draw,fill=black]  (11) {}
      ( 0,0)node[shape=circle,draw,fill=black]    (8) {}
      (1,0)node[shape=circle,draw,fill=black]     (9) {}
      (0,.7)node[shape=circle,draw,fill=black]    (4) {}
      (1,.7)node[shape=circle,draw,fill=black]    (5) {}
      (-0.6,.7)node[shape=circle,draw,fill=black] (6) {}
      (1.6,.7)node[shape=circle,draw,fill=black]  (7) {}
      (.5,1.3)node[shape=rectangle,label=left:\Large{$v_{0}$},draw,fill=red] (0) {}
      (-0.1,1.65)node[shape=circle,draw,fill=black]    (1) {}
      (1,1.65)node[shape=circle,draw,fill=black] (2) {}
      (1.6,2)node[shape=circle,draw,fill=black]  (3) {};

      \draw[-](0)--(1);
      \draw[-](0)--(2);
      \draw[-](2)--(3);
      \draw[-](0)--(4);
      \draw[-](0)--(5);
      \draw[-](4)--(6);
      \draw[-](5)--(7);
      \draw[-](4)--(8);
      \draw[-](5)--(9);
      \draw[-](8)--(10);
      \draw[-](9)--(11);
     
\end{tikzpicture}
}

\newcommand{\dois}{
\begin{tikzpicture}[scale=.7,auto=left,every node/.style={circle,scale=0.5}]

\path ( 0,-0.7)node[shape=circle,draw,fill=black] (10) {}
      (1,-0.7)node[shape=circle,draw,fill=black]  (11) {}
      ( 0,0)node[shape=circle,draw,fill=black]    (8) {}
      (1,0)node[shape=circle,draw,fill=black]     (9) {}
      (0,.7)node[shape=circle,draw,fill=black]    (4) {}
      (1,.7)node[shape=circle,draw,fill=black]    (5) {}
      (.5,1.3)node[shape=rectangle,label=left:\Large{$v_{0}$},draw,fill=red] (0) {}
      (-0.1,1.65)node[shape=circle,draw,fill=black]    (1) {}
      (1,1.65)node[shape=circle,draw,fill=black] (2) {}
      (1.6,2)node[shape=circle,draw,fill=black]  (3) {};

      \draw[-](0)--(1);
      \draw[-](0)--(2);
      \draw[-](2)--(3);
      \draw[-](0)--(4);
      \draw[-](0)--(5);
      \draw[-](4)--(8);
      \draw[-](5)--(9);
      \draw[-](8)--(10);
      \draw[-](9)--(11);
     
\end{tikzpicture}
}

\newcommand{\tres}{
\begin{tikzpicture}[scale=.7,auto=left,every node/.style={circle,scale=0.5}]

\path ( 0,-0.7)node[shape=circle,draw,fill=black] (10) {}
      (1,-0.7)node[shape=circle,draw,fill=black]  (11) {}
      ( 0,0)node[shape=circle,draw,fill=black]    (8) {}
      (1,0)node[shape=circle,draw,fill=black]     (9) {}
      (0,.7)node[shape=circle,draw,fill=black]    (4) {}
      (1,.7)node[shape=circle,draw,fill=black]    (5) {}
      (1.6,.7)node[shape=circle,draw,fill=black]  (7) {}
      (.5,1.3)node[shape=rectangle,label=left:\Large{$v_{0}$},draw,fill=red] (0) {}
      (-0.1,1.65)node[shape=circle,draw,fill=black]    (1) {}
      (1,1.65)node[shape=circle,draw,fill=black] (2) {}
      (1.6,2)node[shape=circle,draw,fill=black]  (3) {};

      \draw[-](0)--(1);
      \draw[-](0)--(2);
      \draw[-](2)--(3);
      \draw[-](0)--(4);
      \draw[-](0)--(5);
      \draw[-](5)--(7);
      \draw[-](4)--(8);
      \draw[-](5)--(9);
      \draw[-](8)--(10);
      \draw[-](9)--(11);
     
\end{tikzpicture}
}

\newcommand{\quatro}{
\begin{tikzpicture}[scale=.7,auto=left,every node/.style={circle,scale=0.5}]

\path ( 0,-0.7)node[shape=circle,draw,fill=black] (10) {}
      (1,-0.7)node[shape=circle,draw,fill=black]  (11) {}
      ( 0,0)node[shape=circle,draw,fill=black]    (8) {}
      (1,0)node[shape=circle,draw,fill=black]     (9) {}
      (0,.7)node[shape=circle,draw,fill=black]    (4) {}
      (1,.7)node[shape=circle,draw,fill=black]    (5) {}
      (-0.6,.7)node[shape=circle,draw,fill=black] (6) {}
      (1.6,.7)node[shape=circle,draw,fill=black]  (7) {}
      (.5,1.3)node[shape=rectangle,label=left:\Large{$v_{0}$},draw,fill=red] (0) {}
      (1,1.65)node[shape=circle,draw,fill=black] (2) {}
      (1.6,2)node[shape=circle,draw,fill=black]  (3) {};

      \draw[-](0)--(2);
      \draw[-](2)--(3);
      \draw[-](0)--(4);
      \draw[-](0)--(5);
      \draw[-](4)--(6);
      \draw[-](5)--(7);
      \draw[-](4)--(8);
      \draw[-](5)--(9);
      \draw[-](8)--(10);
      \draw[-](9)--(11);
     
\end{tikzpicture}
}

\newcommand{\cinco}{
\begin{tikzpicture}[scale=.7,auto=left,every node/.style={circle,scale=0.5}]

\path ( 0,-0.7)node[shape=circle,draw,fill=black] (10) {}
      (1,-0.7)node[shape=circle,draw,fill=black]  (11) {}
      ( 0,0)node[shape=circle,draw,fill=black]    (8) {}
      (1,0)node[shape=circle,draw,fill=black]     (9) {}
      (0,.7)node[shape=circle,draw,fill=black]    (4) {}
      (1,.7)node[shape=circle,draw,fill=black]    (5) {}
      (.5,1.3)node[shape=rectangle,label=left:\Large{$v_{0}$},draw,fill=red] (0) {}
      (1,1.65)node[shape=circle,draw,fill=black] (2) {}
      (1.6,2)node[shape=circle,draw,fill=black]  (3) {};

      \draw[-](0)--(2);
      \draw[-](2)--(3);
      \draw[-](0)--(4);
      \draw[-](0)--(5);
      \draw[-](4)--(8);
      \draw[-](5)--(9);
      \draw[-](8)--(10);
      \draw[-](9)--(11);
     
\end{tikzpicture}
}

\newcommand{\seis}{
\begin{tikzpicture}[scale=.7,auto=left,every node/.style={circle,scale=0.5}]

\path ( 0,-0.7)node[shape=circle,draw,fill=black] (10) {}
      (1,-0.7)node[shape=circle,draw,fill=black]  (11) {}
      ( 0,0)node[shape=circle,draw,fill=black]    (8) {}
      (1,0)node[shape=circle,draw,fill=black]     (9) {}
      (0,.7)node[shape=circle,draw,fill=black]    (4) {}
      (1,.7)node[shape=circle,draw,fill=black]    (5) {}
      (1.6,.7)node[shape=circle,draw,fill=black]  (7) {}
      (.5,1.3)node[shape=rectangle,label=left:\Large{$v_{0}$},draw,fill=red] (0) {}
      (1,1.65)node[shape=circle,draw,fill=black] (2) {}
      (1.6,2)node[shape=circle,draw,fill=black]  (3) {};

      \draw[-](0)--(2);
      \draw[-](2)--(3);
      \draw[-](0)--(4);
      \draw[-](0)--(5);
      \draw[-](5)--(7);
      \draw[-](4)--(8);
      \draw[-](5)--(9);
      \draw[-](8)--(10);
      \draw[-](9)--(11);
     
\end{tikzpicture}
}

\newcommand{\sete}{
\begin{tikzpicture}[scale=.7,auto=left,every node/.style={circle,scale=0.5}]

\path ( 0,-0.7)node[shape=circle,draw,fill=black] (10) {}
      (1,-0.7)node[shape=circle,draw,fill=black]  (11) {}
      ( 0,0)node[shape=circle,draw,fill=black]    (8) {}
      (1,0)node[shape=circle,draw,fill=black]     (9) {}
      (0,.7)node[shape=circle,draw,fill=black]    (4) {}
      (1,.7)node[shape=circle,draw,fill=black]    (5) {}
      (-0.6,.7)node[shape=circle,draw,fill=black] (6) {}
      (1.6,.7)node[shape=circle,draw,fill=black]  (7) {}
      (.5,1.3)node[shape=rectangle,label=left:\Large{$v_{0}$},draw,fill=red] (0) {}
      (-0.1,1.65)node[shape=circle,draw,fill=black]    (1) {}
      ;

      \draw[-](0)--(1);
      \draw[-](0)--(4);
      \draw[-](0)--(5);
      \draw[-](4)--(6);
      \draw[-](5)--(7);
      \draw[-](4)--(8);
      \draw[-](5)--(9);
      \draw[-](8)--(10);
      \draw[-](9)--(11);
     
\end{tikzpicture}
}

\newcommand{\oito}{
\begin{tikzpicture}[scale=.7,auto=left,every node/.style={circle,scale=0.5}]

\path ( 0,-0.7)node[shape=circle,draw,fill=black] (10) {}
      (1,-0.7)node[shape=circle,draw,fill=black]  (11) {}
      ( 0,0)node[shape=circle,draw,fill=black]    (8) {}
      (1,0)node[shape=circle,draw,fill=black]     (9) {}
      (0,.7)node[shape=circle,draw,fill=black]    (4) {}
      (1,.7)node[shape=circle,draw,fill=black]    (5) {}
      (.5,1.3)node[shape=rectangle,label=left:\Large{$v_{0}$},draw,fill=red] (0) {}
      (-0.1,1.65)node[shape=circle,draw,fill=black]    (1) {}
      ;

      \draw[-](0)--(1);
      \draw[-](0)--(4);
      \draw[-](0)--(5);
      \draw[-](4)--(8);
      \draw[-](5)--(9);
      \draw[-](8)--(10);
      \draw[-](9)--(11);
     
\end{tikzpicture}
}

\newcommand{\nove}{
\begin{tikzpicture}[scale=.7,auto=left,every node/.style={circle,scale=0.5}]

\path ( 0,-0.7)node[shape=circle,draw,fill=black] (10) {}
      (1,-0.7)node[shape=circle,draw,fill=black]  (11) {}
      ( 0,0)node[shape=circle,draw,fill=black]    (8) {}
      (1,0)node[shape=circle,draw,fill=black]     (9) {}
      (0,.7)node[shape=circle,draw,fill=black]    (4) {}
      (1,.7)node[shape=circle,draw,fill=black]    (5) {}
      (1.6,.7)node[shape=circle,draw,fill=black]  (7) {}
      (.5,1.3)node[shape=rectangle,label=left:\Large{$v_{0}$},draw,fill=red] (0) {}
      (-0.1,1.65)node[shape=circle,draw,fill=black]    (1) {}
      ;

      \draw[-](0)--(1);
      \draw[-](0)--(4);
      \draw[-](0)--(5);
      \draw[-](5)--(7);
      \draw[-](4)--(8);
      \draw[-](5)--(9);
      \draw[-](8)--(10);
      \draw[-](9)--(11);
     
\end{tikzpicture}
}

\newcommand{\dez}{
\begin{tikzpicture}[scale=.7,auto=left,every node/.style={circle,scale=0.5}]

\path ( 0,-0.7)node[shape=circle,draw,fill=black] (10) {}
      (1,-0.7)node[shape=circle,draw,fill=black]  (11) {}
      ( 0,0)node[shape=circle,draw,fill=black]    (8) {}
      (1,0)node[shape=circle,draw,fill=black]     (9) {}
      (0,.7)node[shape=circle,draw,fill=black]    (4) {}
      (1,.7)node[shape=circle,draw,fill=black]    (5) {}
      (-0.6,.7)node[shape=circle,draw,fill=black] (6) {}
      (1.6,.7)node[shape=circle,draw,fill=black]  (7) {}
      (.5,1.3)node[shape=rectangle,label=left:\Large{$v_{0}$},draw,fill=red] (0) {}
      ;

      \draw[-](0)--(4);
      \draw[-](0)--(5);
      \draw[-](4)--(6);
      \draw[-](5)--(7);
      \draw[-](4)--(8);
      \draw[-](5)--(9);
      \draw[-](8)--(10);
      \draw[-](9)--(11);
     
\end{tikzpicture}
}

\newcommand{\onze}{
\begin{tikzpicture}[scale=.7,auto=left,every node/.style={circle,scale=0.5}]

\path ( 0,-0.7)node[shape=circle,draw,fill=black] (10) {}
      (1,-0.7)node[shape=circle,draw,fill=black]  (11) {}
      ( 0,0)node[shape=circle,draw,fill=black]    (8) {}
      (1,0)node[shape=circle,draw,fill=black]     (9) {}
      (0,.7)node[shape=circle,draw,fill=black]    (4) {}
      (1,.7)node[shape=circle,draw,fill=black]    (5) {}
      (.5,1.3)node[shape=rectangle,label=left:\Large{$v_{0}$},draw,fill=red] (0) {}
      ;

      \draw[-](0)--(4);
      \draw[-](0)--(5);
      \draw[-](4)--(8);
      \draw[-](5)--(9);
      \draw[-](8)--(10);
      \draw[-](9)--(11);
     
\end{tikzpicture}
}

\newcommand{\doze}{
\begin{tikzpicture}[scale=.7,auto=left,every node/.style={circle,scale=0.5}]

\path ( 0,-0.7)node[shape=circle,draw,fill=black] (10) {}
      (1,-0.7)node[shape=circle,draw,fill=black]  (11) {}
      ( 0,0)node[shape=circle,draw,fill=black]    (8) {}
      (1,0)node[shape=circle,draw,fill=black]     (9) {}
      (0,.7)node[shape=circle,draw,fill=black]    (4) {}
      (1,.7)node[shape=circle,draw,fill=black]    (5) {}
      (1.6,.7)node[shape=circle,draw,fill=black]  (7) {}
      (.5,1.3)node[shape=rectangle,label=left:\Large{$v_{0}$},draw,fill=red] (0) {}
      ;

      \draw[-](0)--(4);
      \draw[-](0)--(5);
      \draw[-](5)--(7);
      \draw[-](4)--(8);
      \draw[-](5)--(9);
      \draw[-](8)--(10);
      \draw[-](9)--(11);
     
\end{tikzpicture}
}

\begin{document}
\DeclarePairedDelimiter\ceil{\lceil}{\rceil}
\DeclarePairedDelimiter\floor{\lfloor}{\rfloor}

\title[On eigenvalues of trees of diameter seven]{On the minimum number of eigenvalues of trees of diameter seven}
\author[L. E. Allem]{L. Emilio Allem}
\address{UFRGS - Universidade Federal do Rio Grande do Sul, 
Instituto de Matem\'atica, Porto Alegre, Brazil}\email{emilio.allem@ufrgs.br}
\author[C. Hoppen]{Carlos Hoppen}
\address{UFRGS - Universidade Federal do Rio Grande do Sul, 
Instituto de Matem\'atica, Porto Alegre, Brazil}\email{choppen@ufrgs.br}
\author[L. S. Sibemberg]{Lucas Siviero Sibemberg}
\address{UFRGS - Universidade Federal do Rio Grande do Sul, 
Instituto de Matem\'atica e Estat\'{i}stica, Porto Alegre, Brazil}\email{lucas.siviero@ufrgs.br}



\maketitle

\begin{abstract}
The underlying graph $G$ of a symmetric matrix $M=(m_{ij})\in \mathbb{R}^{n\times n}$ is the graph with vertex set $\{v_1,\ldots,v_n\}$ such that a pair $\{v_i,v_j\}$ with $i\neq j$ is an edge if and only if $m_{ij}\neq 0$. Given a graph $G$, let $q(G)$ be the minimum number of distinct eigenvalues in a symmetric matrix whose underlying graph is $G$. A symmetric matrix $M$ is said to be a realization of $q(G)$ if it has underlying graph $G$ and $q(G)$ distinct eigenvalues.  
In the case of trees, a paper by Johnson and Saiago [Johnson, C.R, and Saiago, C.M, Diameter Minimal Trees, \emph{Linear and Multilinear Algebra} \textbf{64(3)} (2015), 557--571.] proposed an approach by which realizations of large trees are constructed from realizations of smaller trees with the same diameter, known as seeds, which has proved to be very successful. In this paper, we discuss realizations of $q(T)$ for trees of diameter seven based on the seed that defines it, correcting a result in the aforementioned paper.
\end{abstract}


\section{Introduction}\label{sec:introducao}

With a (simple) graph $G=(V,E)$ with vertex set $V=\{v_1,\ldots,v_n\}$, we associate the set of real symmetric matrices $$\mathcal{S}(G) = \{M\in \mathbb{R}^{n\times n} \colon M = M^T, (\forall i \neq j) \left[m_{ij}\neq 0 \Longleftrightarrow \{v_i, v_j\} \in E\right]\},$$
namely the set of real symmetric matrices whose nonzero off-diagonal entries are precisely the entries corresponding to the edges of $G$. There is no constraint on diagonal entries. We say that $G$ is the \emph{underlying graph} of a matrix $M \in \mathcal{S}(G)$.
A parameter that has caught the attention of numerous researchers (see~\cite{ahmadi2013,barrett2020,fallat2022} and their references) is $q(G)$, the minimum number of distinct eigenvalues over all matrices in $\mathcal{S}(G)$, i.e., the minimum degree of the minimal polynomial of any matrix in $\mathcal{S}(G)$. More than computing $q(G)$, there has been interest in explicitly constructing a matrix $M\in \mathcal{S}(G)$ for which the set $\DSpec(M)$ of distinct eigenvalues of $M$ satisfies $|\DSpec(M)|=q(G)$, which we call a \emph{realization} of $q(G)$.

In this paper, we address this problem for trees. Following the terminology of previous work~\cite{JSdiminimal,JohnsonSaiago2018,leal2002minimum}, the diameter $d(T)$ of a tree $T$ is the number of \emph{vertices} on a longest path in $T$\footnote{In the graph theory community, it is more common to define the diameter of a tree $T$ as the number of \emph{edges} on a longest path in $T$. Since we shall work with trees with fixed diameter, it is easy to translate our statements with respect to this other definition.}. 
It is not hard to show (see Leal-Duarte and Johnson~\cite{leal2002minimum}) that, for any tree $T$, the inequality $q(T)\geq d(T)$ holds and that it is tight for some trees. Moreover, researchers such as Barioli and Fallat~\cite{barioli2004two} found examples of trees $T$ for which $q(T) > d(T)$. Hence, it is natural to characterize those trees $T$ for which $q(T)=d(T)$, which are called \emph{diameter minimal trees} (or \emph{diminimal trees}, for short). More generally, for any tree $T$, Johnson and Saiago~\cite{JSdiminimal} defined its \emph{disparity} $D(T)=q(T)-d(T)$, so that $D(T)=0$ if and only if $T$ is diminimal. Also because of~\cite{barioli2004two}, we know that, for any positive integer $k$, there is a tree $T$ such that $D(T)>k$. On the other hand,~\cite[Theorem 4]{JSdiminimal} shows that, for any $d \geq 1$, the quantity 
$$C(d)=\max\{q(T) \colon d(T)=d\}$$
is well defined.

For any fixed integer $d \geq 1$, it is easy to see that the set $\mathcal{D}_d$ of diminimal trees of diameter $d$ is nonempty, as the $d$-vertex path $P_{d}$ lies in $\mathcal{D}_d$. Johnson and Saiago~\cite{JSdiminimal,JohnsonSaiago2018} have shown that the families $\mathcal{D}_{d}$ are infinite for every $d$. In particular, Theorem 3 in~\cite{JSdiminimal} states that all trees of diameter $d \leq 6$ are diminimal, thus $D(T)=0$ for all trees such that $d(T)\leq 6$. 
This is not true for trees of diameter $d=7$, as shown in~\cite{barioli2004two}.  

A full characterization of diminimal trees of diameter $d \geq 7$ is still elusive. To be better able to construct diminimal trees, Johnson and Saiago~\cite{JSdiminimal} defined a decomposition of the set of trees of any fixed diameter $d$ based on a finite set of irreducible trees, known as \emph{seeds}, and on an operation, known as \emph{combinatorial branch duplication}, which we now describe. Given a tree $T$, we say that $U$ is a \textit{branch} of $T$ at a vertex $v$ if $U$ is a connected component of $T-v$. We say that $U$ is rooted at the vertex $u \in V(U)$ that is adjacent to $v$ in $T$. An $s$-\textit{combinatorial branch duplication} ($s$-CBD, for short) of a branch $U$ of a tree $T$ at vertex $v$ results in a new tree $T'$ where $s$ copies of $U$ are appended to $T$ at $v$. A tree $\Tilde{T}$ obtained from a tree $T$ through a sequence of CBDs is called an \textit{unfolding} of $T$. An example is depicted in Figure~\ref{fig:unfolding0}. We shall restrict our attention to $s$-CBDs that keep the diameter unchanged, as done in~\cite{allem2023diminimal,JSdiminimal,JohnsonSaiago2018}. A \textit{seed} is a tree that is not an unfolding of any smaller tree of the same diameter. For example, there is a single seed of diameter $d$ if $d\in \{3,4\}$, there are two seeds of diameter $5$, three seeds of diameter $6$, and twelve seeds of diameter $7$. The latter are depicted in Figure~\ref{fig:seeds_seis}. 

\begin{figure}
\centering
\begin{minipage}[h]{.5\textwidth}
\centering 
\begin{subfigure}{0.49\textwidth}
\centering 
     \begin{tikzpicture}[scale=.7,auto=center,every node/.style={circle,scale=0.5}]

\path(0,-.2)node[shape=circle,draw,fill=black] (1) {}
     (1,-.2)node[shape=circle,draw,fill=black] (2) {}
     (.5,.5)node[shape=circle,draw,fill=black] (3) {}
      (.5,1.35)node[shape=circle,label=left:\Large{$u$},draw,fill=black] (0) {}
      (.5,2.2)node[shape=rectangle,label=left:\Large{$v$},draw,fill=red] (4) {}
      (1.2,2.2)node[shape=circle,draw,fill=black] (5) {}
      (1.9,2.2)node[shape=circle,draw,fill=black] (6) {}
      (2.6,2.2)node[shape=circle,draw,fill=black] (7) {}
      
      (0.5,2.9)node[shape=circle,draw,fill=black] (8) {};

     \draw[-](1)--(3);
     \draw[-](2)--(3);
     \draw[-](3)--(0);
     \draw[-](4)--(0);
     \draw[-](5)--(4);
     \draw[-](5)--(6);
     \draw[-](7)--(6);
     \draw[-](4)--(8);

\end{tikzpicture}
     \caption{Tree $T$ of diameter $7$ rooted at $v$.}
 \end{subfigure}
 
 \end{minipage}\hfill
\begin{minipage}[h]{.5\textwidth}
\centering

 \begin{subfigure}{0.49\textwidth}
     \centering
\begin{tikzpicture}[scale=.7,auto=left,every node/.style={circle,scale=0.5}]

\path(0,-.2)node[shape=circle,draw,fill=black] (1) {}
     (1,-.2)node[shape=circle,draw,fill=black] (2) {}
     (.5,.5)node[shape=circle,draw,fill=black] (3) {}
      (.5,1.35)node[shape=circle,label=left:\Large{$u$},draw,fill=black] (0) {}
      (.5,2.2)node[shape=rectangle,label=left:\Large{$v$},draw,fill=red] (4) {}
      (1.2,2.2)node[shape=circle,draw,fill=black] (5) {}
      (1.9,2.2)node[shape=circle,draw,fill=black] (6) {}
      (2.6,2.2)node[shape=circle,draw,fill=black] (7) {}
      
      (0.5,2.9)node[shape=circle,draw,fill=black] (8) {}

      (2,1.35)node[shape=circle,draw,fill=black] (00) {}
      (1.5,-.2)node[shape=circle,draw,fill=black] (10) {}
      (2.5,-.2)node[shape=circle,draw,fill=black] (11) {}
      (2,.5)node[shape=circle,draw,fill=black] (12) {}

      (3.5,1.35)node[shape=circle,draw,fill=black] (000) {}
      (3,-.2)node[shape=circle,draw,fill=black] (13) {}
      (4,-.2)node[shape=circle,draw,fill=black] (14) {}
      (3.5,.5)node[shape=circle,draw,fill=black] (15) {};

     \draw[-](1)--(3);
     \draw[-](2)--(3);
     \draw[-](3)--(0);
     \draw[-](4)--(0);
     \draw[-](5)--(4);
     \draw[-](5)--(6);
     \draw[-](7)--(6);
     \draw[-](8)--(4);

     \draw[-](11)--(12);
     \draw[-](10)--(12);
     \draw[-](00)--(12);
     \draw[dotted](00)--(4);
     
     \draw[-](13)--(15);
     \draw[-](14)--(15);
     \draw[-](000)--(15);
     \draw[dotted](000)--(4);
     
\end{tikzpicture}     
\caption{$2$-CBD of $U$ (the branch of $T-v$ that contains $u$) at $v$.}
 \end{subfigure}
 
 \end{minipage}
\caption{The tree in (B) is an unfolding of the tree $T$ in (A).}\label{fig:unfolding0}
\end{figure}
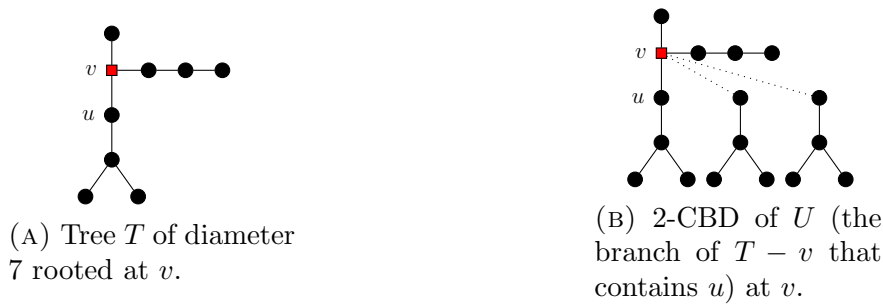


\captionsetup[subfigure]{labelformat=empty}
\begin{figure}
\centering
\input{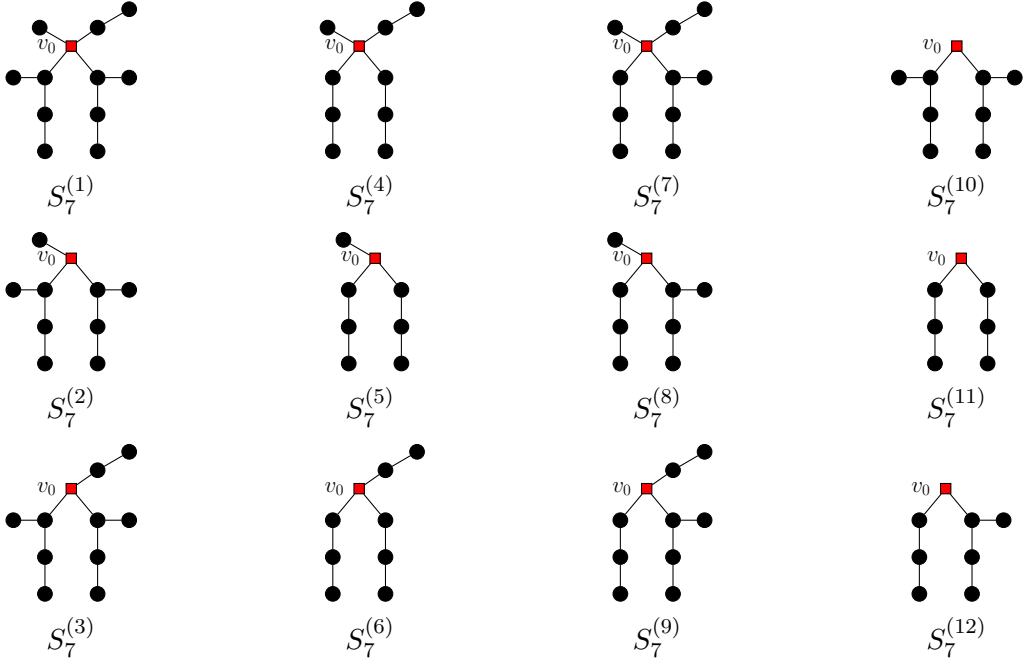}
\caption{The twelve seeds of diameter $7$, denoted $S_7^{(1)},\ldots,S_7^{(12)}$.}
\label{fig:seeds_seis}
\end{figure}

As shown in~\cite{JSdiminimal} (see also \cite{JohnsonSaiago2018,Proyecciones2018}), two of the main reasons to consider seeds 
are the property that, for every $d\in\mathbb{N}$, there exist finitely many seeds of diameter $d$, and the property that every tree with diameter $d$ may be obtained as the unfolding of a unique seed of diameter $d$. In addition to this, Braga, Oliveira, Trevisan and the current authors~\cite{allem2023diminimal} have shown that, for any given diameter $d$, there exists a seed of diameter $d$ with the property that all of its unfoldings are diminimal. A seed with this property is called a \emph{diminimal seed} and, prior to~\cite{allem2023diminimal}, the authors of~\cite{JohnsonSaiago2018} had shown that any seed of diameter at most five is diminimal. On the other hand, seeds for which at least one unfolding is not diminimal will be called \emph{defective seeds}. 

The results of Barioli and Fallat~\cite{barioli2004two} imply that the seeds $S^{(10)}_7$, $S^{(11)}_7$ and $S^{(12)}_7$ in Figure~\ref{fig:seeds_seis} are defective. 
Given a seed $S$ of diameter $d$, let $\mathcal{T}(S)$ be the set of unfoldings of $S$ of diameter $d$. The authors of~\cite{JSdiminimal} claimed that all remaining nine seeds of diameter seven were diminimal, and, for each such seed $S_7^{(i)}$, they depicted the structure of a realization in~\cite[Appendix 1]{JSdiminimal} (see also~\cite[Appendix B]{JohnsonSaiago2018}) that purportedly could be extended to a realization of any tree in $\mathcal{T}(S_7^{(i)})$. Unfortunately, while the realizations of six of the nine seeds could indeed be extended as explained in~\cite{JSdiminimal}, the realizations described for $S^{(7)}_7$, $S^{(8)}_7$ and $S^{(9)}_7$ cannot be extended in the way that was suggested. As it turns out, the seeds are actually defective. This is summarized as follows.
\begin{theorem}\label{t1}
For diameter 7, the seed $S^{(i)}_7$ depicted in Figure~\ref{fig:seeds_seis} is diminimal if and only if $i \leq 6$.
\end{theorem}

Moreover, the authors of~\cite{JohnsonSaiago2018} have relied on their incorrect conclusion about $S^{(7)}_7$, $S^{(8)}_7$ and $S^{(9)}_7$ being diminimal to prove that $C(7)=8$ (see~\cite[Theorem 4]{JSdiminimal}). This statement is correct, but the full proof depends on the result below. 
\begin{theorem}\label{t2}
For any tree $T \in \mathcal{T}(S^{(7)}_7) \cup \mathcal{T}(S^{(8)}_7) \cup \mathcal{T}(S^{(9)}_7)$, it holds that $q(T) \leq 8$.
\end{theorem}

The remainder of the paper is structured as follows. In Section~\ref{sec:literature}, we describe an eigenvalue location algorithm by Jacobs and Trevisan~\cite{JT2011} that may be applied to weighted trees and is the main tool in our proofs. 
In Section~\ref{sec:thm11}, we show that the seeds $S^{(i)}_7$ are defective for $7 \leq i \leq 9$, which, together with the results in~\cite{barioli2004two}, establishes that $S^{(i)}_7$ is defective for every $i \geq 7$. The fact that the seeds $S^{(i)}_7$ are diminimal for $i \leq 6$ is a consequence of~\cite{JSdiminimal}.

To conclude the paper, in Section~\ref{sec:thm2}, we show that Theorem~\ref{t2} holds by presenting matrices for  $S^{(7)}_7$, $S^{(8)}_7$ and $S^{(8)}_7$ with eight distinct eigenvalues that can be extended to realizations of any tree in $\mathcal{T}(S^{(7)}_7) \cup \mathcal{T}(S^{(8)}_7) \cup \mathcal{T}(S^{(9)}_7)$.



%

\newcommand{\Tzero}{
\begin{tikzpicture}[scale=.7,auto=left,every node/.style={circle,scale=0.5}]
\path( 0,0)node[shape=circle,draw,fill=black] (1) {}
      (1,0)node[shape=circle,draw,fill=black] (2) {}
      (0,.7)node[shape=circle,draw,fill=black] (3) {}
      (1,.7)node[shape=circle,draw,fill=black] (4) {}
      (.5,1.3)node[shape=rectangle,label=left:\Large{$v_{0}$},draw,fill=red] (5) {}
      (.5,2)node[shape=circle,draw,fill=black] (6) {};

      \draw[-](1)--(3);
      \draw[-](3)--(5);
      \draw[-](2)--(4);
      \draw[-](4)--(5);
      \draw[-](5)--(6);
\end{tikzpicture}
}
\newcommand{\Tdois}{
\begin{tikzpicture}[scale=.7,auto=left,every node/.style={circle,scale=0.5}]
\path( 0,0)node[shape=circle,draw,fill=black] (1) {}
      (1,.5)node[shape=rectangle,label=right:\Large{$v_{2}$},draw,fill=red] (2) {};       

      \draw[-](1)--(2);
          
     \end{tikzpicture}
}
\newcommand{\Tum}{
\begin{tikzpicture}[scale=.7,auto=left,every node/.style={circle,scale=0.5}]

\path( 0,0)node[shape=circle,draw,fill=black] (1) {}
      (.5,.5)node[shape=rectangle,label=left:\Large{$v_{1}$},draw,fill=red] (2) {}
      (1,1)node[shape=circle,draw,fill=black] (3) {};
      \draw[-](1)--(2);
      \draw[-](2)--(3);  
\end{tikzpicture}
}

\newcommand{\operation}{
\begin{tikzpicture}[scale=.7,auto=left,every node/.style={circle,scale=0.5}]
\path( 0,0)node[shape=circle,draw,fill=black] (1) {}
      (.6,0)node[shape=circle,draw,fill=black] (2) {}
      (.3,1)node[shape=circle,label=left:\Large{$v_{1}$},draw,fill=black] (3) {} 

      (1.5,0)node[shape=circle,draw,fill=black] (4) {}
      (1.5,1)node[shape=circle,label=left:\Large{$v_{2}$},draw,fill=black] (5) {}  

      (2.5,0)node[shape=circle,draw,fill=black] (6) {}
      (2.5,1)node[shape=circle,label=left:\Large{$v_{3}$},draw,fill=black] (7) {} 
      (3.2,0)node[shape=circle,draw,fill=black] (8) {}      
      (3.5,1)node[shape=circle,draw,fill=black] (9) {}
      (3.7,0)node[shape=circle,draw,fill=black] (10) {}

      (1.5,2)node[shape=rectangle,label=left:\Large{$v_{0}$},draw,fill=red] (0) {}
      (1.2,3)node[shape=circle,draw,fill=black] (11) {}
      (2.2,2.5)node[shape=circle,draw,fill=black] (14) {}
      (3,2.5)node[shape=circle,draw,fill=black] (15) {}
      (2.2,3.5)node[shape=circle,draw,fill=black] (12) {}
      (3,3.5)node[shape=circle,draw,fill=black] (13) {};

      \draw[-](1)--(3);
      \draw[-](2)--(3);

      \draw[-](4)--(5);

      \draw[-](6)--(7);
      \draw[-](7)--(9);
      \draw[-](8)--(9);
      \draw[-](9)--(10);

      \draw[-](0)--(11);
      \draw[-](0)--(14);
      \draw[-](0)--(12);
      \draw[-](14)--(15);
      \draw[-](12)--(13);

      \draw[dotted](0)--(3);
      \draw[dotted](0)--(5);
      \draw[dotted](0)--(7);
\end{tikzpicture}
}

\section{Eigenvalue location in trees}\label{sec:literature}

The main technical tool in the proofs of this paper is an eigenvalue location algorithm.  
For a real symmetric matrix $M$ whose underlying graph is a tree and a real number $x$, Algorithm \texttt{Diagonalize} (see Figure~\ref{treealgo}) finds a diagonal matrix $D$ that is \emph{congruent} to $M+Ix$. This means that there is an invertible matrix $S$ such that $D=S(M+Ix)S^T$. By Sylvester's Law of Inertia, the multiplicity of $-x$ as an eigenvalue of $M$ is equal to the number of zeros on the diagonal of $D$, the number of eigenvalues (with multiplicity) of $M$ that are greater than $-x$ is equal to the number of positive entries of $D$, and the number of eigenvalues (with multiplicity) of $M$ that are less than $-x$ is equal to the number of negative entries of $D$. This is why it is known as an eigenvalue location algorithm. For more information, we refer to the original paper by Jacobs and Trevisan~\cite{JT2011}, which focused on the adjacency matrix of trees, to the paper by Braga and Rodrigues~\cite{TEMA1041} that extended it to arbitrary symmetric matrices whose underlying graph is a tree, and to the book~\cite{diagonalize}, which considers eigenvalue location more broader classes of matrices.

The algorithm orders the vertices on the tree in a particular way. To define this ordering, consider a tree $T=(V,E)$ rooted at vertex $r$. As usual, we say that $T$ has \emph{depth} $k$ if $k$ is the length of a longest path between $r$ and a leaf of $T$. The level of a vertex $v$ on the tree is the distance between $v$ and $r$, so that we may partition $V$ as $V=B_0 \cup \cdots \cup B_k$, where $B_i$ denotes the set of vertices of $T$ at level $i$. The ordering $v_1,\ldots,v_n$ of $V$ used by the algorithm is such that vertices in $B_k$ appear first, vertices in $B_{k-1}$ appear next, and so on, up to $v_n=r$. 

Coming back to the algorithm, each vertex is assigned an initial value, and the algorithm works in rounds, so that each vertex $v_i$ is \emph{processed} in one of the rounds. Processing a vertex might change the value assigned to it and the value assigned to at most one of its children.

\begin{figure}
{\tt
\begin{tabbing}
aaa\=aaa\=aaa\=aaa\=aaa\=aaa\=aaa\=aaa\= \kill
     \> Input: matrix $M = (m_{ij})$ with underlying tree $T$\\
     \> Input: Bottom up ordering $v_1,\ldots,v_n$ of $V(T)$\\
     \> Input: real number $x$ \\
     \> Output: diagonal matrix $D = \mbox{diag}(d_1, \ldots, d_n)$ congruent to $M + xI$ \\
     \> \\
     \>  Algorithm $\mbox{Diagonalize}(M, x)$ \\
     1\> \> initialize $d_i := m_{ii} + x$, for all $i$ \\
     2\> \> {\bf for } $k = 1$ to $n$ \\
     3\> \> \> {\bf if} $v_k$ is a leaf {\bf then} continue \\
     4\> \> \> {\bf else if} $d_c \neq 0$ for all children $c$ of $v_k$ {\bf then} \\
     5\> \> \>  \>   $d_k := d_k - \sum \frac{(m_{ck})^2}{d_c}$, summing over all children of $v_k$ \\
     6\> \> \> {\bf else } \\
     7\> \> \> \> select one child $v_j$ of $v_k$ for which $d_j = 0$  \\
     8\> \> \> \> $d_k  := -\frac{(m_{jk})^2}{2}$ \\
     9\> \> \> \> $d_j  :=  2$ \\
     10\> \> \> \> if $v_k$ has a parent $v_\ell$, remove the edge $\{v_k,v_\ell\}$. \\
     11\> \>  {\bf end loop} \\
\end{tabbing}
}
\caption{Algorithm Diagonalize}
\label{treealgo}
\end{figure}

We also state the application of Sylvester's Law of Inertia to the output of Algorithm~\ref{treealgo} as a formal result.
\begin{theorem}
\label{thm_localizacao}
Let $M$ be a symmetric matrix and let $x$ be a real number. Let $D$ be the diagonal matrix produced by Algorithm~\ref{treealgo} with input $M$ and $x$. The following hold:
\begin{itemize}
\item[(a)] The number of eigenvalues of $M$ that are greater than $-x$ is equal to the number of positive entries in the diagonal of $D$.
\item[(b)] The number of eigenvalues of $M$ that are less than $-x$ is equal to the number of negative entries in the diagonal of $D$.
\item[(c)] The multiplicity of $-x$ as an eigenvalue of $M$ is equal to the number of zeros in the diagonal of $D$.
\end{itemize}
\end{theorem}

Based on the algorithm, we may also establish the following well-known result. For a proof, see~\cite[Theorem 2.2]{allem2023diminimal}.
\begin{theorem}\label{thm:simpleroots}
Let $T$ be a tree with root $r$, let $M\in \mathcal{S}(T)$, and consider $\lambda_{\min}=\min(\Spec(M))$ and $\lambda_{\max}=\max(\Spec(M))$. Then, $m_{M}(\lambda_{\min}) = 1 =m_{M}(\lambda_{\max})$. Moreover, $r$ is the only vertex $v$ for which the algorithm \texttt{Diagonalize}$(M,-\lambda)$ assigns $d_v=0$.
\end{theorem}

\section{Proof of Theorem~\ref{t1}}\label{sec:thm11}

The aim of this section is to show that the seeds $S_7^{(7)}$, $S_7^{(8)}$ and $S_7^{(9)}$ are defective. We start with a proof for $S_7^{(8)}$, the other two proofs will use a similar strategy. Recall that a rooted tree has depth $k$ if $k$ is the length of the path between the root and the furthest leaf. Given a tree $T$ rooted at a vertex $r$, we define $B_j=\{u\in V(T):d(u,r)=k-j \}$ as the sets of vertices of $T$ that are  $k-j$ edges away from $r$. We say that vertices in $B_j$ are at \emph{level} $j$. This terminology means that the levels are considered bottom-up (i.e., $B_0$ contains the furthest leaves and $B_k$ contains the root), which is not standard, but is consistent with the way in which algorithm \texttt{Diagonalize} acts on the tree.

\begin{defn}\label{def:mkl}
     Let $T$ be a rooted tree of depth $k$, $\lambda\in\mathbb{R}$ and $M\in\mathcal{S}(T)$. Given $j \in\{0,1,\ldots,k\}$, consider $T_j=B_0\cup\cdots\cup B_j$. For any integer $\ell$, where $0\leq \ell \leq j$, we define $m^{(M)}_{j,\ell}(\lambda)$ as the number of $0$'s at level $\ell$ when we apply \texttt{Diagonalize}$(M[T_j],-\lambda)$. Let $m^{(M)}_{j,\ell}=\sum_{\lambda\in\mathbb{R}}m^{(M)}_{j,\ell}(\lambda)$.
\end{defn}

This definition leads to
\begin{eqnarray}
|V(T_j)| = \sum_{\lambda\in\mathbb{R}}m_{M[T_j]}(\lambda) &=& \sum_{\lambda\in\mathbb{R}}\sum_{\ell=0}^j m^{(M)}_{j,\ell}(\lambda)=\sum_{\ell=0}^j m^{(M)}_{j,\ell}\label{auxeq1}.
\end{eqnarray}
Indeed, given $\lambda\in\mathbb{R}$, $\sum_{\ell=0}^j m^{(M)}_{j,\ell}(\lambda)$ is the total number of $0$'s in the diagonal of the diagonal matrix produced by $\texttt{Diagonalize}(M[T_j],-\lambda)$, which by Theorem~\ref{thm_localizacao}(c) is the multiplicity of $\lambda$ in $M[T_j]$.

Further note that, for any $\lambda\in\mathbb{R}$ such that $m^{(M)}_{k-1,k-1}(\lambda)>0$, we have $m^{(M)}_{k,k-1}(\lambda)=m^{(M)}_{k-1,k-1}(\lambda)-1$, because when running the algorithm $\texttt{Diagonalize}(M,-\lambda)$ on the layer $B_{k}=\{r\}$ that contains the root of $T$, one of the zeros at a child of $r$ will be redefined as $2$. As a consequence, for 
$$N_T(M)=N(M)=\left|\{\lambda\in\mathbb{R}:m^{(M)}_{k-1,k-1}(\lambda)>0\}\right|,$$ we have $m^{(M)}_{k,k-1}=m^{(M)}_{k-1,k-1}-N(M)$. This allows us to prove the following.

\begin{lemat}\label{lemmamkk}
    Let $T$ be a rooted tree of depth $k$ and $M\in\mathcal{S}(T)$. Then, $m^{(M)}_{k,k}=N(M)+1$.
\end{lemat}

\begin{proof}
Let $r$ be the root of $T$. We have
    \begin{eqnarray*}
        n = \sum_{j=0}^k m^{(M)}_{k,j} &=& m^{(M)}_{k,k}+m^{(M)}_{k,k-1}+\sum_{j=0}^{k-2}m^{(M)}_{k,j}\\
        &=& m^{(M)}_{k,k}+m^{(M)}_{k-1,k-1}-N(M)+\sum_{j=0}^{k-2}m^{(M)}_{k-1,j}\\
        &=& m^{(M)}_{k,k}-N(M)+\sum_{j=0}^{k-1}m^{(M)}_{k-1,j}\\
        &=& m^{(M)}_{k,k}-N(M)+|V(T-r)|\\
        &=& m^{(M)}_{k,k}-N(M)+(n-1)
    \end{eqnarray*}
    This leads to $m^{(M)}_{k,k}=N(M)+1$.
\end{proof}

From Lemma~\ref{lemmamkk} we can obtain a lower bound on $m^{(M)}_{k,k}$, i.e., on the quantity of distinct real numbers $\lambda$ such that $\texttt{Diagonalize}(M,-\lambda)$ assigns $0$ to the root of $M$.
\begin{lemat}\label{no_zero}
    Let $T$ be a rooted tree of depth $k$ and $M\in\mathcal{S}(T)$. Then, 
    $$m^{(M)}_{k,k}= \left|\{\lambda\in\mathbb{R}:m^{(M)}_{k,k}(\lambda)>0\}\right| \geq  k+1.$$ 
\end{lemat}

Before proving Lemma~\ref{no_zero}, note that $m^{(M)}_{k,j}$ is not necessarily equal to $|\{\lambda\in\mathbb{R}:m^{(M)}_{k,j}(\lambda)>0\}|$, since $m^{(M)}_{k,j}$ counts the total number of zeros at level $j$ assigned by \texttt{Diagonalize}$(M,-\lambda)$ for all $\lambda\in \mathbb{R}$, while $\left|\{\lambda\in\mathbb{R}:m^{(M)}_{k,j}(\lambda)>0\}\right|$ is the number of distinct $\lambda \in \mathbb{R}$ such that \texttt{Diagonalize}$(M,-\lambda)$ produces some zero at level $j$. However, for a tree $T$ of depth $k$, $m^{(M)}_{k,k}= \left|\{\lambda\in\mathbb{R}:m^{(M)}_{k,k}(\lambda)>0\}\right|$, given that there is a single vertex of $T$ at level $k$, namely its root. 

\begin{proof}[Proof of Lemma~\ref{no_zero}]
    We prove by induction on $k$. For $k=0$, we have $M=[\alpha]$ for some $\alpha \in \mathbb{R}$, so that $\lambda_1=-\alpha$ is such that $m^{(M)}_{0,0} = m^{(M)}_{0,0}(-\alpha) = 1$.

    Suppose that the statement is valid for some $k\in\mathbb{N}$. Let $T$ be an $n$-vertex tree of depth $k+1$ rooted at a vertex $r$, and consider $M\in\mathcal{S}(T)$.
    Let $U_1,\ldots,U_p$ be the connected components of $T-r$ rooted, respectively, at the neighbor $u_i\in V(U_i)$ of $r$. Since $T$ has depth $k+1$, at least one of the connected components has depth $k$. Without loss of generality, let $U_1$ be a component with depth $k$.
    By definition, $\{\lambda\in\mathbb{R}:m^{(M[U_1])}_{k,k}(\lambda)>0\}\subseteq\{\lambda\in\mathbb{R}:m^{(M)}_{k,k}(\lambda)>0\}$. By the induction hypothesis on $U_1$, we have that $m^{(M[U_1])}_{k,k}\geq k+1$ and, since $B_{k}(U_1)=\{u_1\}$ contains a single vertex, if $\lambda \in \mathbb{R}$ is such that $m^{(M[U_1])}_{k,k}(\lambda)>0$, then $m^{(M[U_1])}_{k,k}(\lambda)=1$. This implies that $\left|\{\lambda\in\mathbb{R}:m^{(M[U_1])}_{k,k}(\lambda)>0\}\right|\geq k+1$, which establishes

    \begin{equation*}
    k+1\leq\left|\{\lambda\in\mathbb{R}:m^{(M[U_1])}_{k,k}(\lambda)>0\}\right|\leq N(M).
    \end{equation*}
   Lemma~\ref{lemmamkk} tells us that $m^{(M)}_{k+1,k+1} = N(M)+1\geq k+2$, which concludes the induction. 
\end{proof}

We are now ready to show that $S_7^{(7)}$, $S_7^{(8)}$ and $S_7^{(9)}$ are defective. In Figure~\ref{counter_example} we depict a tree $T$ that is an unfolding of the seed $S_7^{(8)}$ (see Figure~\ref{fig:seeds_seis}) and that is not diminimal.

\begin{figure}[H]
    \centering
    \hspace{-1cm}
    \input{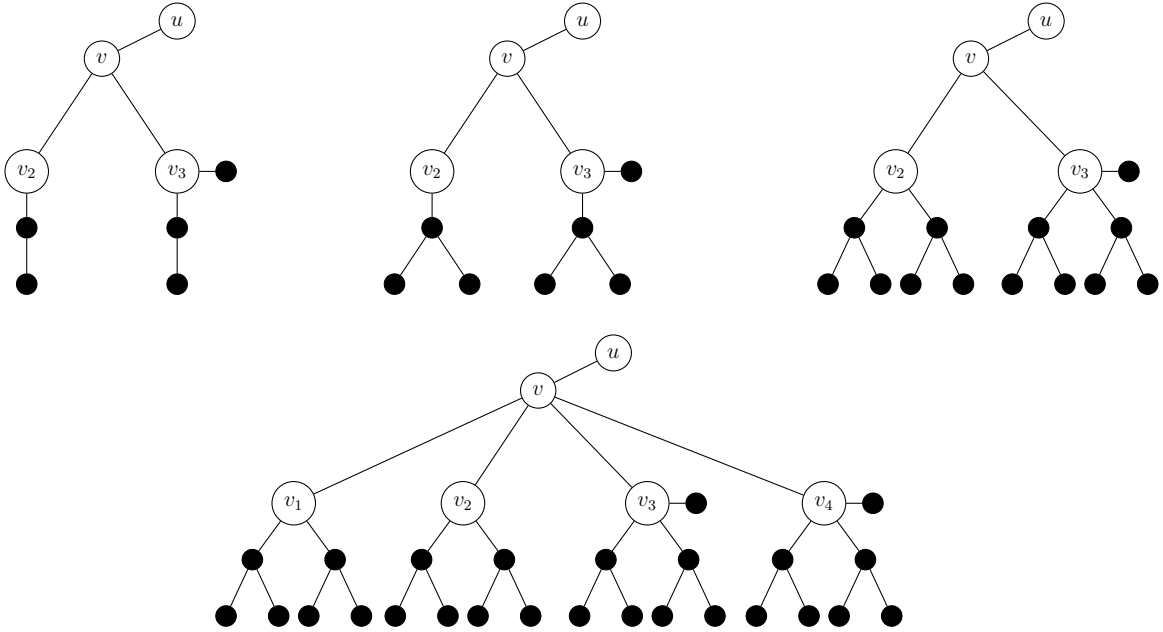}
    \caption{The seed $S_7^{(8)}$ and a sequence of diameter-preserving unfoldings that produce a tree $T$ that is not diminimal.}
    \label{counter_example}
\end{figure}


  \begin{theorem}\label{badseedS68}
      The seed $S_7^{(8)}$ is a defective seed.
  \end{theorem}

  \begin{proof}
Consider the tree $T$ of diameter $7$ depicted in Figure~\ref{counter_example}, viewed as a tree rooted at vertex $v$. Let $v_i$, for $i\in [4]=\{1,2,3,4\}$, be the (non-leaf) children of $v$, as they appear in Figure~\ref{counter_example}. Let $T_i$ be the connected component of $T-v$ that contains $v_i$, viewed as a rooted tree with root $v_i$, for $i\in [4]$. Let $u$ be the leaf connected to $v$. Let $\overline{T}=T_1\cup T_2\cup T_3\cup T_4$, so that $V(\overline{T})=V-\{u,v\}$.

Let $M\in\mathcal{S}(T)$. For brevity, we write $\sigma(M)=\Spec(M)$ in this proof. For any $\lambda\in\sigma(M)$, let $L(M,\lambda)$ be the distance between the main root and the closest vertex to the main root for which \texttt{Diagonalize}$(M,-\lambda)$ assigns value $0$. We also define $\sigma_j=\{\lambda\in\sigma(M[\overline{T}]):L(M[\overline{T}],\lambda)=j\}$. We start with a claim about the eigenvalues of $M$.
            
\begin{claim}\label{claim1}
The following are true:
\begin{enumerate}
    \item [(i)] If $\lambda \in \sigma(M)$ does not lie in $\sigma(M[\overline{T}])$, then $m_M(\lambda)=1$.
    \item [(ii)] If $\lambda\in\sigma_j$ for $j>0$, then $\lambda\in\sigma(M)$ with $m_M(\lambda)=m_{M[\overline{T}]}(\lambda)$ or $m_M(\lambda)=m_{M[\overline{T}]}(\lambda)+1$.
    \item [(iii)] If $\lambda\in\sigma_0$, then $m_M(\lambda)=m_{M[\overline{T}]}(\lambda)$ if $m_{uu}=\lambda$ and $m_M(\lambda)=m_{M[\overline{T}]}(\lambda)-1$ otherwise.
\end{enumerate}
\end{claim}

\begin{proof}[Proof of Claim]
Consider an application of $\texttt{Diagonalize}(M,-\lambda)$ using an ordering for which $u$ and $v$ are the last two vertices. Before processing $u$ and $v$, each remaining vertex $w$ is assigned the same value $d_w$ that would be assigned to it by $\texttt{Diagonalize}(M[T_i],-\lambda)$ for the corresponding component $T_i$. 

For item $(i)$, suppose that $\lambda \notin \sigma(M[\overline{T}])$, then all values assigned by $\texttt{Diagonalize}(M,-\lambda)$ to all vertices of $\overline{T}$ is nonzero. By Theorem~\ref{thm_localizacao}(c) and given that $\lambda \in \sigma(M)$, we have $1 \leq m(\lambda) \leq 2$. Moreover, $m_M(\lambda)=2$ if and only if $d_u$ and $d_v$ are both zero at the end of the algorithm. This is not possible, as any vertex for which a child has value 0 is assigned a negative value by the algorithm. 

In the item $(ii)$, given that $\lambda\in \sigma_j$ with $j>0$, there are vertices $w \in V(\overline{T})$ for which $d_w=0$, but these vertices are not $v_1,\ldots,v_4$. As a consequence, the values assigned to these vertices will not change while processing $u$ and $v$. As explained in the previous paragraph, the algorithm may not assign value 0 to both $u$ and $v$. By Theorem~\ref{thm_localizacao}(c), $m_T(\lambda) \in \{m_{\overline{T}}(\lambda),m_{\overline{T}}(\lambda)+1\}$. 

For the item $(iii)$ we suppose that $\lambda\in \sigma_0$. Then, there are vertices $w \in V(\overline{T})$ for which $d_w=0$, at least one of which is $v_1,\ldots,v_4$. Next we process the leaf $u$, which has value $d_u=m_{uu}-\lambda$.

If $m_{uu}=\lambda$, we have $d_u=0$. While processing $v$, the value of one of its children (say $d_u$) is redefined as 2, and $d_v$ is assigned a negative number. Thus no new zero is produced, and $m_M(\lambda)=m_{M[\overline{T}]}(\lambda)$ by Theorem~\ref{thm_localizacao}(c).

If $m_{uu}\neq \lambda$, we have $d_u\neq 0$. While processing $v$, the value of one of the $v_i$ for which $d_{v_i}=0$) is redefined as 2, and $d_v$ is assigned a negative number. This means that the number of zero values decreases by 1 and $m_M(\lambda)=m_{M[\overline{T}]}(\lambda)-1$ by Theorem~\ref{thm_localizacao}(c). 
\end{proof}

Coming back to our proof, let $\lambda_{\min}$ and $\lambda_{\max}$ denote the minimum and the maximum eigenvalue of $M$, respectively. By Theorem~\ref{thm:simpleroots}, we know that their multiplicity in $M$ is 1 and that $\lambda_{\min},\lambda_{\max}\notin\sigma(M[\overline{T}])$. 

Since $T_i$ has diameter $5$ for each $i \in [4]$, we know that $M[T_i]$ has at least five distinct eigenvalues. Moreover, by Lemma~\ref{no_zero}, given that $T_i$ has depth 2, there exist at least three distinct real numbers $\mu_1,\mu_2,\mu_3$ such that $L(M[T_i],\mu_\ell)=0$ for all $\ell\in [3]$. 
Any eigenvalue of $M[T_i]$ that is not an eigenvalue of $M$ is called a \emph{missing eigenvalue} of branch $T_i$.

To reach a contradiction, suppose that $|\DSpec(M)|=7$. Let the distinct eigenvalues of $M$ be $\lambda_{\min}<\lambda_1<\lambda_2<\lambda_3<\lambda_4<\lambda_5<\lambda_{\max}$. Our proof will proceed as follows. First we show that $\lambda_1,\ldots,\lambda_5$ must be eigenvalues of $\sigma(M[\overline{T}])$. Then, we show that the set $\{\lambda_1,\ldots,\lambda_5\}$ is precisely $\sigma(M[\overline{T}])$. Finally we obtain a system of equations involving the eigenvalues $\lambda_j$ for the different components $T_i$. 
The desired contradiction comes from the fact that the system is infeasible.

It is important to observe that $|V(\overline{T})|=n-2=m_{M}(\lambda_1)+m_{M}(\lambda_2)+m_{M}(\lambda_3)+m_{M}(\lambda_4)+m_{M}(\lambda_5)$. If each branch $T_i$ has $s_i$ missing eigenvalues, we get
\begin{equation}\label{eq30}
\sum_{\lambda\in\sigma(M[\overline{T}])\cap\sigma(M)}m_{M[\overline{T}]}(\lambda)=|V(\overline{T})|-(s_1+\cdots+s_4),
\end{equation}
where the equality comes from the fact that the multiplicity of any missing eigenvalue must be exactly $1$ in $\overline{T}$, because if for a missing eigenvalue $\lambda$, $m_{M[\overline{T}]}(\lambda)\geq2$, then, by Claim~\ref{claim1}(i), we have $\lambda\in\sigma(M)$.

Furthermore, since $\lambda_{\min},\lambda_{\max}\notin\sigma(M[\overline{T}])$, there are at most five distinct values that are in $\sigma_j$, for some $j\geq 0$ and in $\sigma(M)$, that is, $\left|\sigma(M)\cap\sigma(M[\overline{T}])\right|\leq 5$. Let $B=\{\lambda_1,\lambda_2,\lambda_3,\lambda_4,\lambda_5\}\setminus\sigma(M[\overline{T}])$ and $j=|B|$. 
We observe that any $\lambda_\ell \in B$ must be a simple eigenvalue of $M$ by Claim~\ref{claim1}(i).

Note that
\begin{eqnarray}\label{eq31}    
        |V(\overline{T})| 
        &=& m_{M}(\lambda_1)+m_{M}(\lambda_2)+m_{M}(\lambda_3)+m_{M}(\lambda_4)+m_{M}(\lambda_5)\nonumber\\\nonumber
        &=& \sum_{\lambda\in B} m_{M}(\lambda) + \sum_{\lambda\in\sigma_0\cap\sigma(M)} m_{M} (\lambda) + \sum_{\lambda\in\sigma_{\geq1}\cap\sigma(M)} m_{M} (\lambda)\\
        &=& j + \sum_{\lambda\in\sigma_0\cap\sigma(M)} m_{M} (\lambda) + \sum_{\lambda\in\sigma_{\geq1}\cap\sigma(M)} m_{M} (\lambda)
\end{eqnarray}
Claim~\ref{claim1}(ii) implies that
$$\sum_{\lambda\in\sigma_{\geq1}\cap\sigma(M)} m_{M} (\lambda) \leq \sum_{\lambda\in\sigma_{\geq1}} [m_{M[\overline{T}]} (\lambda)+1].$$
By Claim~\ref{claim1}(iii), we have
$$\sum_{\lambda\in\sigma_0\cap\sigma(M)} m_{M} (\lambda) \leq \delta_{u} + \sum_{\lambda\in\sigma_0\cap\sigma(M)} (m_{M[\overline{T}]}(\lambda)-1),$$
where $\delta_{u}=1$ if $m_{uu}$ lies in the set $\{\lambda \colon \lambda \in \sigma_0\cap\sigma(M)\}$ and $\delta_{u}=0$ otherwise.

As a consequence, \eqref{eq31} leads to
\begin{eqnarray}
        |V(\overline{T})|&\leq& j+\delta_u + \sum_{\lambda\in\sigma(M[\overline{T}])\cap\sigma(M)}m_{M[\overline{T}]}(\lambda) - |\sigma_0\cap\sigma(M)| +|\sigma_{\geq1}\cap\sigma(M)| \nonumber \\
        &\leq& j+\delta_u + |V(\overline{T})| - (s_1+\cdots+s_4) - |\sigma_0\cap\sigma(M)| +|\sigma_{\geq1}\cap\sigma(M)|. \label{eq32}
        \end{eqnarray}

Summing $2|\sigma_{0}\cap\sigma(M)|$ to both sides of \eqref{eq32} and rearranging the terms, we get
\begin{equation}\label{eq333}
|\sigma_{0}\cap\sigma(M)|+|\sigma_{\geq1}\cap\sigma(M)| \geq 2|\sigma_0\cap\sigma(M)| + (s_1+\cdots+s_4)-j-\delta_u.
\end{equation}        
Moreover, given that $\delta_u\leq 1$, we have that
\begin{equation}\label{eq33}
|\sigma_{0}\cap\sigma(M)|+|\sigma_{\geq1}\cap\sigma(M)| \geq 2|\sigma_0\cap\sigma(M)| + (s_1+\cdots+s_4)-j-1.
\end{equation}
\begin{claim}\label{claim4}
 The set $B=\{\lambda_1,\ldots,\lambda_5\}\setminus \sigma(M[\overline{T}])$ is empty.
 \end{claim} 

\begin{proof}[Proof of Claim] Towards a contradiction, assume that $j\geq 1$, so that $|\sigma(M)\cap\sigma(M[T_i])|\leq 5-j<5$. In particular, because $M[T_i]$ has at least five distinct eigenvalues for each branch $T_i$, we have $s_i \geq j$ for each $i$, so that~\eqref{eq33} becomes
\begin{equation}\label{eq34}
|\sigma_0\cap\sigma(M)|+|\sigma_{\geq1}\cap\sigma(M)| \geq 2|\sigma_0\cap\sigma(M)| + 3j -1.
\end{equation} 
Since $|\sigma_0\cap\sigma(M)|+|\sigma_{\geq1}\cap\sigma(M)|=|\sigma(M)\cap\sigma(M[T_i])|= 5-j$,~\eqref{eq34} cannot hold for $j\geq 2$, so assume that $j \leq 1$. Because we are also assuming $j \geq 1$, inequality \eqref{eq34} is equivalent to
\begin{equation}\label{eq35}
5>4=5-j=|\sigma_0\cap\sigma(M)|+|\sigma_{\geq1}\cap\sigma(M)| \geq 2|\sigma_0\cap\sigma(M)| + 2.
\end{equation}
For \eqref{eq35} to hold, we must have $|\sigma_0\cap\sigma(M)|\leq 1$.

By Lemma~\ref{no_zero}, since each $T_i$ has depth 2, there must be at least three values $\gamma$ such that $L(M[T_i],\gamma)=0$. So, for each $T_i$, there are at least two values in $\sigma(M[T_i]) \setminus \sigma(M)$, that is, $s_i \geq 2$ for each $i$, so that~\eqref{eq33} leads to
\begin{equation*}
5>|\sigma_0\cap\sigma(M)|+|\sigma_{\geq1}\cap\sigma(M)| \geq 2|\sigma_0\cap\sigma(M)| + 6,
\end{equation*}
a contradiction. 
\end{proof}

At this point, we know that $B=\emptyset$ and so $\{\lambda_{1}\ldots,\lambda_{5}\}\subseteq\sigma(M[\overline{T}])$. Our focus now is to show that $\{\lambda_{1}\ldots\lambda_{5}\}=\sigma(M[\overline{T}])$. Moving further, since $B=\emptyset$, the inequality~\eqref{eq33} becomes
\begin{equation}\label{eq36}
5 \geq |\sigma_0\cap\sigma(M)| +|\sigma_{\geq1}\cap\sigma(M)| \geq 2|\sigma_0\cap\sigma(M)| + (s_1+\cdots+s_4)-1.
\end{equation}
As in the previous paragraph, if $|\sigma_0\cap\sigma(M)|\leq 1$, we get $s_i \geq 2$ for every $i$, and~\eqref{eq36} leads to
$$5 \geq |\sigma_0\cap\sigma(M)| + |\sigma_{\geq1}\cap\sigma(M)| \geq 2 |\sigma_0\cap\sigma(M)|+8-1,$$ 
which is impossible.

Thus $|\sigma_0\cap\sigma(M)|\geq 2$. If $|\sigma_0\cap\sigma(M)|= 2$, we again get $s_i \geq 1$ for every $i$, and 
~\eqref{eq36} leads to
$$5 \geq |\sigma_0\cap\sigma(M)| +|\sigma_{\geq1}\cap\sigma(M)| \geq 2|\sigma_0\cap\sigma(M)|+4-1 \geq 7,$$ 
another contradiction. As a consequence, we must have $|\sigma_0\cap\sigma(M)|\geq 3$, so that $|\sigma_0\cap\sigma(M)|=3$ and
$|\sigma_{\geq 1} \cap\sigma(M)|=2$ because~\eqref{eq36} leads to
$$2 \geq 5- |\sigma_0\cap\sigma(M)| \geq  |\sigma_{\geq1}\cap\sigma(M)| \geq |\sigma_0\cap\sigma(M)| -1 \geq 2.$$

Next we claim that $\sigma_0=\sigma_0 \cap \sigma(M)$. If this is not the case, one of the eigenvalues in $\sigma_0$ must be missing at some branch $T_i$, so that $s_1+\cdots+s_4 \geq 1$. Inequality~\eqref{eq36} becomes $5 \geq 2 \cdot 3 + 1 - 1=6$, another contradiction.

Algorithm \texttt{Diagonalize} ensures that any eigenvalue of $\sigma_{\geq 1}$ must be in $\sigma(M)$. So, we are in the case where the eigenvalues of $M(\overline{T})$ are precisely $\{\lambda_1,\ldots,\lambda_5\}$, and exactly three of them are in $\sigma_0$. In particular, $\DSpec(M[T_i])=\{\lambda_1,\ldots,\lambda_5\}$ and $\{\lambda \colon L(M[T_i],\lambda)=0\}=\{\lambda_1,\dots,\lambda_5\} \cap \sigma_0$ for every $i \in [4]$.   

We shall now consider the trees $T_i$, $i \in [4]$, starting with $T_1$ and $T_2$, which are isomorphic. These are trees with seven vertices, and we refer to each vertex as in Figure~\ref{T1} (on the left).

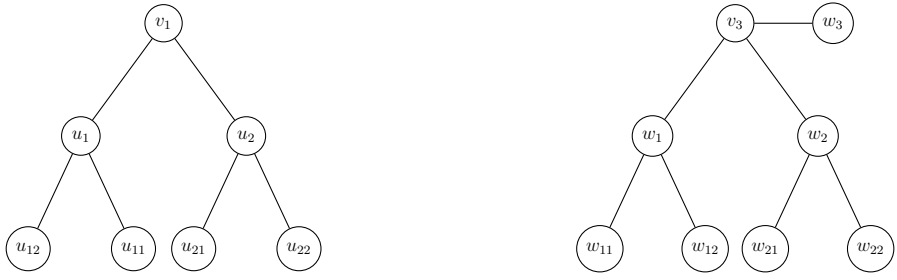
\begin{figure}
\centering
\begin{tikzpicture}[scale=2,auto=left,every node/.style={circle,scale=0.6}]

\path    
     (-1   ,-0.5)node[shape=circle,draw,fill=white] (0) {$v_1$}
     (-1.55,-1.25)node[shape=circle,draw,fill=white] (01) {$u_{1}$}
     (-0.45,-1.25)node[shape=circle,draw,fill=white] (02) {$u_{2}$}
     (-1.2 ,-2)node[shape=circle,draw,fill=white]    (011) {$u_{11}$}
     (-1.9   ,-2)node[shape=circle,draw,fill=white]  (012) {$u_{12}$}
     (-0.1    ,-2)node[shape=circle,draw,fill=white] (021) {$u_{22}$}
     (-0.8 ,-2)node[shape=circle,draw,fill=white]    (022) {$u_{21}$}
     
      (-1   +3.8,-0.5)node[shape=circle,draw,fill=white] (3) {$v_3$}
     (-0.35 +3.8,-0.5)node[shape=circle,draw,fill=white]  (30) {$w_{3}$}
     (-1.55 +3.8,-1.25)node[shape=circle,draw,fill=white] (31) {$w_{1}$}
     (-0.45 +3.8,-1.25)node[shape=circle,draw,fill=white] (32) {$w_{2}$}
     (-1.2  +3.8,-2)node[shape=circle,draw,fill=white]    (311) {$w_{12}$}
     (-1.9  +3.8  ,-2)node[shape=circle,draw,fill=white]  (312) {$w_{11}$}
     (-0.1  +3.8   ,-2)node[shape=circle,draw,fill=white] (321) {$w_{22}$}
     (-0.8  +3.8,-2)node[shape=circle,draw,fill=white]    (322) {$w_{21}$};

      \draw[-](0)--(01);
      \draw[-](0)--(02);
      \draw[-](01)--(011);
      \draw[-](01)--(012);
      \draw[-](02)--(021);
      \draw[-](02)--(022);

      \draw[-](3)-- (30);
      \draw[-](3)-- (31);
      \draw[-](3)-- (32);
      \draw[-](31)--(311);
      \draw[-](31)--(312);
      \draw[-](32)--(321);
      \draw[-](32)--(322);

     \end{tikzpicture}
            \caption{Notation for $T_1$ and $T_2$ (on the left) and for $T_3$ and $T_4$ (on the right).}
            \label{T1}
        \end{figure}
        
    For $T_1$ and $T_2$, the only possible multiplicity lists are $(3,1,1,1,1)$ and $(2,2,1,1,1)$. We will focus on $T_1$. First, consider the list $(2,2,1,1,1)$. Let $\lambda$ be an eigenvalue with multiplicity $2$. At the end of Diag$(M[T_i],-\lambda)$, we must have two zeros in the diagonal of the output matrix. Given that a vertex and its parent cannot both be assigned 0, the possibilities are:
        \begin{itemize}
            \item[(i)] The leaves $u_{11}$, $u_{12}$, $u_{21}$ and $u_{22}$ start with $0$ and, when their parents are processed, two zeros become 2 and the other two remain zero. For this to happen, the entries corresponding to $u_{11}$, $u_{12}$, $u_{21}$ and $u_{22}$ in the diagonal of $M$ must have value $\lambda$. The entry corresponding to $v_1$ cannot be $\lambda$.
            
            \item[(ii)] $u_{11}$, $u_{12}$ start with $0$ and, when their parent is processed, we keep one zero. For this to happen, the entries corresponding to $u_{11}$ and  $u_{12}$ in the diagonal of $M$ must have value $\lambda$, while the entry corresponding to $u_{21}$ or $u_{22}$ is not $\lambda$. The other zero must come from the root $v_1$, as the algorithm assigns $-1/2$ to $u_1$. If any vertex other than $v_1$ in the branch of $u_2$ had been assigned value 0, this value would change while processing its parent.
            
            \item[(iii)] $u_{21}$, $u_{22}$ start with $0$ and, when their parent is processed, we keep one zero. The other zero must come from the root $v_1$. The reasoning for this is analogous to the previous case. 
        \end{itemize}

        Note that, if $(i)$ happens, $(ii)$ and $(iii)$ cannot happen. Moreover, $(i)$ can happen for at most one eigenvalue. Since we need two eigenvalues with multiplicity two, $(i)$ cannot happen. Similarly, case $(ii)$ cannot happen for two different eigenvalues, nor can case $(iii)$. The only additional possibility would be for $(ii)$ and $(iii)$ to happen together, each for a different eigenvalue. By Theorem~\ref{thm:simpleroots}, the least and the greatest eigenvalue of $M[T_1]$ are simple and the zero produced by the algorithm must appear at the root. This forces $|\sigma_0| \geq 4$, contradicting the fact that $|\sigma_0|=|\sigma_0\cap \sigma(M)|=3$.

        Next, consider the list $(3,1,1,1,1)$. Assume that the eigenvalues are ordered as $\lambda_1<\cdots<\lambda_5$. There is a single way for an eigenvalue $\lambda$ to have multiplicity $3$ in $T_1$. The leaves $u_{11}$, $u_{12}$, $u_{21}$ and $u_{22}$ must start with $0$ and, when their parents are processed, two zeros become 2 and the other two remain zero (and the parents are assigned the value $-1/2$). For this to happen, the entries corresponding to $u_{11}$, $u_{12}$, $u_{21}$ and $u_{22}$ in the diagonal of $M$ must have value $\lambda$. The additional zero comes from the root $v_1$ if the diagonal has value $\lambda$.

This means that $L(M[T_1],\lambda)=0$. Moreover, since the algorithm produces two negative diagonal values and two positive diagonal values, we conclude that $\lambda=\lambda_3$. Addionally, the minimum and maximum eigenvalues $\lambda_1$ and $\lambda_5$ of $T_1$ are both simple and the zero associated with them appears at the root.  

By our previous discussion, the other two eigenvalues of $M[T_1]$, namely $\lambda_2$ and $\lambda_4$, must satisfy $L(M[T_1],\lambda_2)=L(M[T_1],\lambda_4)=1$. Note that we cannot produce any additional zero at the leaves, as all of them are zero for $\lambda_3$. So, the algorithm must produce 0 for both $u_1$ and $u_2$ (given that one of the zeros becomes 2 when processing $v_1$).

Now we use the fact that the sum of the diagonal of a matrix is equal to the sum of its eigenvalues, to find a constraint. On one hand, the trace of $M[T_1]$ is equal to  $\tr(M[T_1])=\lambda_1+\lambda_2+3\lambda_3+\lambda_4+\lambda_5$. On the other hand, it is equal to $\tr(M[v_1])+\tr(M[u_1,u_{11},u_{12}])+\tr(M[u_2,u_{21},u_{22}])=\lambda_3+(\lambda_2+\lambda_3+\lambda_4)+(\lambda_2+\lambda_3+\lambda_4)$, where the notation $M[x_1,\ldots,x_m])$ stands for the $m\times m$ submatrix of $M$ relative to the rows and columns associated with $x_1,\ldots,x_m$. This leads to the following equation
\begin{equation}\label{eq:rest1}
\lambda_2 + \lambda_4 = \lambda_1+\lambda_5.
\end{equation}

Given that $T_2$ is isomorphic to $T_1$, the matrix $M[T_2]$ must have the same structure as $M[T_1]$.

Now, for $T_3$ and $T_4$, the possible multiplicity lists are $(4,1,1,1,1)$, $(3,2,1,1,1)$ and $(2,2,2,1,1)$. As with $T_1$, we will consider what may happen to get multiplicity 2 as we apply $\texttt{Diagonalize}(M[T_3],-\lambda)$ (we use the notation in Figure~\ref{T1} (on the right)):
        \begin{itemize}
            \item[(i)] The leaves $w_{11}$, $w_{12}$, $w_{21}$ and $w_{22}$ start with $0$ and, when their parents are processed, two zeros become 2 and the other two remain zero. For this to happen, the entries corresponding to $w_{11}$, $w_{12}$, $w_{21}$ and $w_{22}$ in the diagonal of $M$ must have value $\lambda$. An additional 0 may only appear for $v_3$, and it appears if and only if $d_{w_3}=m_{w_3w_3}-\lambda \neq 0$ and $m_{v_3v_3}-\lambda-(m_{v_3w_{3}})^2/d_{w_3}=0$.
            
            \item[(ii)] $w_{11}$, $w_{12}$ start with $0$ and, when their parent is processed, we keep one of them. For this to happen, the entries corresponding to $w_{11}$ and  $w_{12}$ in the diagonal of $M$ must have value $\lambda$, while the entry corresponding to $w_{21}$ or $w_{22}$ is not $\lambda$. There are two possibilities for the additional 0. The first possibility is $w_2$ or $w_3$, which happens if and only if $d_{w_3}=m_{w_3w_3}-\lambda = 0$, $d_{w_{21}},d_{w_{22}}\neq 0$ and $m_{w_2w_2}-\lambda-(m_{w_2w_{21}})^2/d_{w_{21}}-(m_{w_2w_{22}})^2/d_{w_{22}}=0$. The second possibility is $v_3$, which happens if and only if $d_{w_{2}},d_{w_{3}}\neq 0$ and $m_{v_3v_3}-\lambda-(m_{v_3w_{2}})^2/d_{w_{2}}-(m_{v_3w_3})^2/d_{w_{3}}=0$. No additional zero could happen in this case.
            
            \item[(iii)] $w_{21}$, $w_{22}$ start with $0$ and, when their parent is processed, we keep one zero. The other zero must come from one of the vertices in $\{w_1,w_3\}$ or from $v_3$. No additional zero could happen in this case. The reasoning for this is analogous to the previous case. 

           \item[(iv)] $d_{w_{11}},d_{w_{12}},d_{w_{21}},d_{w_{22}} \neq 0$ and the algorithm assigns 0 to $w_1$, $w_2$ and $w_3$. Two of the zeros remain after processing $v_1$. This happens if and only if $m_{w_1w_1}-\lambda-(m_{w_1w_{11}})^2/d_{w_{11}}-(m_{w_1w_{12}})^2/d_{w_{12}}=0$, $m_{w_2w_2}-\lambda-(m_{w_2w_{21}})^2/d_{w_{21}}-(m_{w_2w_{22}})^2/d_{w_{22}}=0$, and $d_{w_3}=m_{w_3w_3}-\lambda = 0$. No additional zero could happen in this case.
        \end{itemize}

First note that the list $(4,1,1,1,1)$ cannot happen because the multiplicity of any eigenvalue $\lambda$ is at most 3. Moreover, it is clear that (i) may each hold for at most one value of $\lambda$, and, if (i) holds, (ii) and (iii) cannot happen for any value of $\lambda$.

Item (ii) can only hold for a single value of $\lambda$. Moreover, if both (ii) and (iv) hold (for distinct values $\lambda$ and $\lambda'$, respectively), then the second $0$ associated with $\lambda$ must happen at the root $v_1$. The same applies to (iii). Moreover, if both (ii) and (iii) hold, the second $0$ associated with at least one of them must happen at the root. Finally, item (iv) may hold for a single value of $\lambda$.

We claim that $(2,2,2,1,1)$ is not a feasible multiplicity list. Since we need three eigenvalues with multiplicity 2, case (i) cannot happen, as it precludes (ii) and (iii). Therefore (ii), (iii) and (iv) must hold for the eigenvalues $\lambda_2$, $\lambda_3$ and $\lambda_4$, given that the least and the greatest eigenvalue of $M[T_3]$ are simple. However, this means that $\lambda_1$, $\lambda_5$ and the two eigenvalues that appear in cases (ii) and (iii) have the property that 0 appears at the root. This contradicts the fact that there are only three eigenvalues with this property.

This means that the multiplicity list must $(3,2,1,1,1)$, so that cases (i) and (iv) must both happen, and (i) must have a 0 at the root. Assume that the eigenvalues associated with (i) and (iv) are $\lambda$ and $\lambda'$, respectively. Because (i) applies, $m_{w_{11}w_{11}}=m_{w_{12}w_{12}}=m_{w_{21}w_{21}}=m_{w_{22}w_{22}}=\lambda$, while (iv) gives $m_{w_3w_3}=\lambda'$. 

First assume that $\lambda<\lambda'$. When applying $\texttt{Diagonalize}(M,-\lambda)$, we get three 0's, two 2's, two $-1/2$'s and $d_{w_3}=m_{w_3w_3}-\lambda=\lambda'-\lambda>0$. This means that $\lambda=\lambda_3$ and $\lambda'=\lambda_4$. In this case, the trace of $M[T_3]$ satisfies  
$\tr(M[T_3])=\lambda_1+\lambda_2+3\lambda_3+2\lambda_4+\lambda_5$. On the other hand, it is equal to $\tr(M[v_3,w_3])+\tr(M[w_1,w_{11},w_{12}])+\tr(M[w_2,w_{21},w_{22}])$. Consider $M[w_1,w_{11},w_{12}]$. When we apply $\texttt{Diagonalize}(M[w_1,w_{11},w_{12}],-\lambda)$ and when we apply $\texttt{Diagonalize}(M[w_1,w_{11},w_{12}],-\lambda'),$ we get one 0 for each, so that $\lambda_3$ and $\lambda_4$ are eigenvalues. The other eigenvalue is some number $\lambda^\ast$ that satisfies
$$m_{w_1w_1}-\lambda^\ast=\frac{(m_{w_1w_{11}})^2+(m_{w_1w_{12}})^2}{\lambda-\lambda^{\ast}}.$$
Similarly, when $\texttt{Diagonalize}(M[w_2,w_{21},w_{22}],-\lambda)$ or $\texttt{Diagonalize}(M[w_2,w_{21},w_{22}],-\lambda')$ is applied, we get one 0 for each, so that $\lambda_3$ and $\lambda_4$ are eigenvalues. The other eigenvalue is some number $\lambda^{\ast\ast}$ that satisfies
$$m_{w_2w_2}-\lambda^{\ast\ast}=\frac{(m_{w_2w_{12}})^2+(m_{w_2w_{22}})^2}{\lambda-\lambda^{\ast\ast}}.$$

\begin{claim}\label{claim5}
It holds that $\lambda^\ast= \lambda^{\ast \ast}$.
\end{claim}

\begin{proof}[Proof of Claim]
Towards a contradiction, assume that $\lambda^\ast\neq \lambda^{\ast \ast}$. We shall apply Lemma~\ref{lemmamkk} to the matrix $M[T_3]$ associated with a tree of depth two, so that we use the notation in this lemma for $k=2$, $M[T_3]$ and 
$$N(M[T_3])=\left|\{\lambda \colon m_{1,1}^{M[T_3]}(\lambda)>0\} \right|.$$
For the definition of $m_{j,\ell}^{(M[T_3])}$, we refer the reader to Definition~\ref{def:mkl}. The number $N(M[T_3])$ records the number of distinct real numbers $\lambda$ such that $\texttt{Diagonalize}(M[T_3-v_3],-\lambda)$ produces a 0 at $w_1$, $w_2$ or $w_3$. By our assumption, the values $\lambda_4,\lambda^\ast,\lambda^{\ast\ast}$ have this property, so that $N(M[T_3])\geq 3$.

By Lemma~\ref{lemmamkk}, $m_{2,2}^{(M[T_3])}=N(M[T_3])+1\geq 4$, which means that $0$ appears at the root of $T_3$ in an application of $\texttt{Diagonalize}(M[T_3],-\lambda)$ for at least four distinct values of $\lambda$. This is a contradiction, as we already know that the values are $\lambda_1$, $\lambda_3$ or $\lambda_5$.
\end{proof}

Claim~\ref{claim5} ensures that $\lambda^\ast=\lambda^{\ast \ast}=\lambda_2$ because an application of $\texttt{Diagonalize}(M,-\lambda^\ast)$ produces a $0$ in $w_1$ or $w_2$, so that $\lambda^\ast \in \Spec(M[T_3])$. Moreover, $\lambda^\ast \notin \{\lambda_1, \lambda_5\}$, since the least and greatest eigenvalue of $M[T_3]$ can only produce a $0$ at the root $v_3$ of $T_3$. This means that $\tr(M(T_3[w_1,w_{11},w_{12}]))=\tr(M(T_3[w_2,w_{21},w_{22}]))=\lambda_2+\lambda_3+\lambda_4$.
Regarding $M(T_3[v_3,w_3])$, we already know that $\lambda=\lambda_3$ is an eigenvalue with multiplicity one and that $\lambda'=\lambda_4$ is not an eigenvalue. Let $\lambda^\ast$ be the second eigenvalue. This leads to 
\begin{eqnarray*}
\lambda_1+\lambda_2+3\lambda_3&+&2\lambda_4+\lambda_5=\tr(M[T_3])\\
&=& \tr(M(T_3[v_3,w_3]))+\tr(M(T_3[w_{1},w_{11},w_{12}]))+\tr(M(T_3[w_{2},w_{21},w_{22}]))\\
&=& 2\lambda_2+3\lambda_3+2\lambda_4+\lambda^\ast
\end{eqnarray*}
This leads to the equation 
\begin{equation}\label{eq:rest2}
\lambda^\ast +\lambda_2=\lambda_1+\lambda_5, \textrm{ where }\lambda^\ast \neq \lambda_4.
\end{equation}
Note that~\eqref{eq:rest1} (which is a necessary condition for $T_1$) and~\eqref{eq:rest2} are mutually exclusive, so that this case cannot happen.

The only remaining case is $\lambda'<\lambda$. When applying $\texttt{Diagonalize}(M,-\lambda)$, we get three 0's, two 2's, two $-1/2$'s and $d_{w_3}=m_{w_3w_3}-\lambda=\lambda'-\lambda<0$. This means that $\lambda=\lambda_3$ and $\lambda'=\lambda_2$. Proceeding as in the previous case, we will reach
\begin{equation}\label{eq:rest3}
\lambda^\ast +\lambda_4=\lambda_1+\lambda_5, \textrm{ where }\lambda^\ast \neq \lambda_2.
\end{equation}
As before,~\eqref{eq:rest1} and~\eqref{eq:rest3} are mutually exclusive. 

This shows that it is impossible to construct $M[\overline{T}]$ such that its eigenvalues are $\{\lambda_1,\ldots,\lambda_5\}$, and exactly three of them are in $\sigma_0$. Thus there is no matrix $M\in \mathcal{S}(T)$ such that $|\DSpec(M)|=7$.  
\end{proof}

This proof has other implications. We say that a graph $G$ whose connected components are $G_1,\ldots,G_s$ has Property $C$ if $q(G) = \max_i q(G_i)$. Note that it is always true that $q(G) \geq \max_i q(G_i)$. Let $\mathcal{C}$ be the class of graphs with Property C. One might think that $\mathcal{C}$ contains all graphs, but the following corollary shows that it is not true.
\begin{coro}
Let $T=T_1\cup T_3$, where $T_1$ and $T_3$ are the trees defined in Figure~\ref{T1}. Then, $T\notin \mathcal{C}$.  
\end{coro}
\begin{proof}
Recall that, in the proof of Theorem~\ref{badseedS68}, our strategy was to show that, assuming that a diminimal realization $M$ of $T$ exists, the distinct eigenvalues of $\sigma(M[\overline{T}])$ must be a set $\{\lambda_1,\ldots,\lambda_5\}$, where $\overline{T}$ is a forest with two components isomorphic to $T_1$ and two components isomorphic to $T_3$. In the end we obtained a system of equations involving the eigenvalues $\lambda_j$ for the different components $T_i$. 
The desired contradiction came from the fact that the system is infeasible, which means that $\lambda_1,\ldots,\lambda_5$ cannot be the distinct eigenvalues of both $M[T_1]$ and $M[T_3]$. This implies that, even though $T_1$ and $T_3$ are both diminimal (since they have diameter 5), we have
$$q(T) \geq 6 > 5 =  \max\{q(T_1),q(T_3)\}.$$
\end{proof}

\newcommand{\counterexampledois}{
\begin{tikzpicture}[scale=1,auto=left,every node/.style={circle,scale=0.7}]

\path(0,1)node[shape=circle,draw,fill=white] (a) {$v$}
     (1,1.5)node[shape=circle,draw,fill=white] (b) {$u$}
     (2,2)node[shape=circle,draw,fill=white] (c) {$w$}
     
     (-1   ,-0.5)node[shape=circle,draw,fill=white] (0) {$v_2$}
     (-1.55,-1.25)node[shape=circle,draw,fill=black] (01) {}
     (-0.45,-1.25)node[shape=circle,draw,fill=black] (02) {}
     (-1.2 ,-2)node[shape=circle,draw,fill=black]    (011) {}
     (-1.9   ,-2)node[shape=circle,draw,fill=black]  (012) {}
     (-0.1    ,-2)node[shape=circle,draw,fill=black] (021) {}
     (-0.8 ,-2)node[shape=circle,draw,fill=black]    (022) {}

     (-3.25   ,-0.5)node[shape=circle,draw,fill=white] (1) {$v_1$}
     (-3.8,-1.25)node[shape=circle,draw,fill=black] (11) {}
     (-2.7,-1.25)node[shape=circle,draw,fill=black] (12) {}
     (-3.45 ,-2)node[shape=circle,draw,fill=black]    (111) {}
     (-4.15   ,-2)node[shape=circle,draw,fill=black]  (112) {}
     (-2.35    ,-2)node[shape=circle,draw,fill=black] (121) {}
     (-3.05 ,-2)node[shape=circle,draw,fill=black]    (122) {}

     (-1   +2.45,-0.5)node[shape=circle,draw,fill=white] (2) {$v_3$}
     (-0.35+2.45,-0.5)node[shape=circle,draw,fill=black]  (20) {}
     (-1.55+2.45,-1.25)node[shape=circle,draw,fill=black] (21) {}
     (-0.45+2.45,-1.25)node[shape=circle,draw,fill=black] (22) {}
     (-1.2 +2.45,-2)node[shape=circle,draw,fill=black]    (211) {}
     (-1.9 +2.45  ,-2)node[shape=circle,draw,fill=black]  (212) {}
     (-0.1 +2.45   ,-2)node[shape=circle,draw,fill=black] (221) {}
     (-0.8 +2.45,-2)node[shape=circle,draw,fill=black]    (222) {}

     (-1   +4.8,-0.5)node[shape=circle,draw,fill=white] (3) {$v_4$}
     (-0.35+4.8,-0.5)node[shape=circle,draw,fill=black]  (30) {}
     (-1.55+4.8,-1.25)node[shape=circle,draw,fill=black] (31) {}
     (-0.45+4.8,-1.25)node[shape=circle,draw,fill=black] (32) {}
     (-1.2 +4.8,-2)node[shape=circle,draw,fill=black]    (311) {}
     (-1.9 +4.8  ,-2)node[shape=circle,draw,fill=black]  (312) {}
     (-0.1 +4.8   ,-2)node[shape=circle,draw,fill=black] (321) {}
     (-0.8 +4.8,-2)node[shape=circle,draw,fill=black]    (322) {} ;

      \draw[-](a)--(b);
      \draw[-](b)--(c);
      \draw[-](a)--(0);
      \draw[-](a)--(1);
      \draw[-](a)--(2);
      \draw[-](a)--(3);

      \draw[-](0)--(01);
      \draw[-](0)--(02);
      \draw[-](01)--(011);
      \draw[-](01)--(012);
      \draw[-](02)--(021);
      \draw[-](02)--(022);

      \draw[-](1)-- (11);
      \draw[-](1)-- (12);
      \draw[-](11)--(111);
      \draw[-](11)--(112);
      \draw[-](12)--(121);
      \draw[-](12)--(122);

      \draw[-](2)-- (20);
      \draw[-](2)-- (21);
      \draw[-](2)-- (22);
      \draw[-](21)--(211);
      \draw[-](21)--(212);
      \draw[-](22)--(221);
      \draw[-](22)--(222);

      \draw[-](3)-- (30);
      \draw[-](3)-- (31);
      \draw[-](3)-- (32);
      \draw[-](31)--(311);
      \draw[-](31)--(312);
      \draw[-](32)--(321);
      \draw[-](32)--(322);
    
     \end{tikzpicture}
}


\begin{figure}[H]
    \centering
    \hspace{-1cm}
    \input{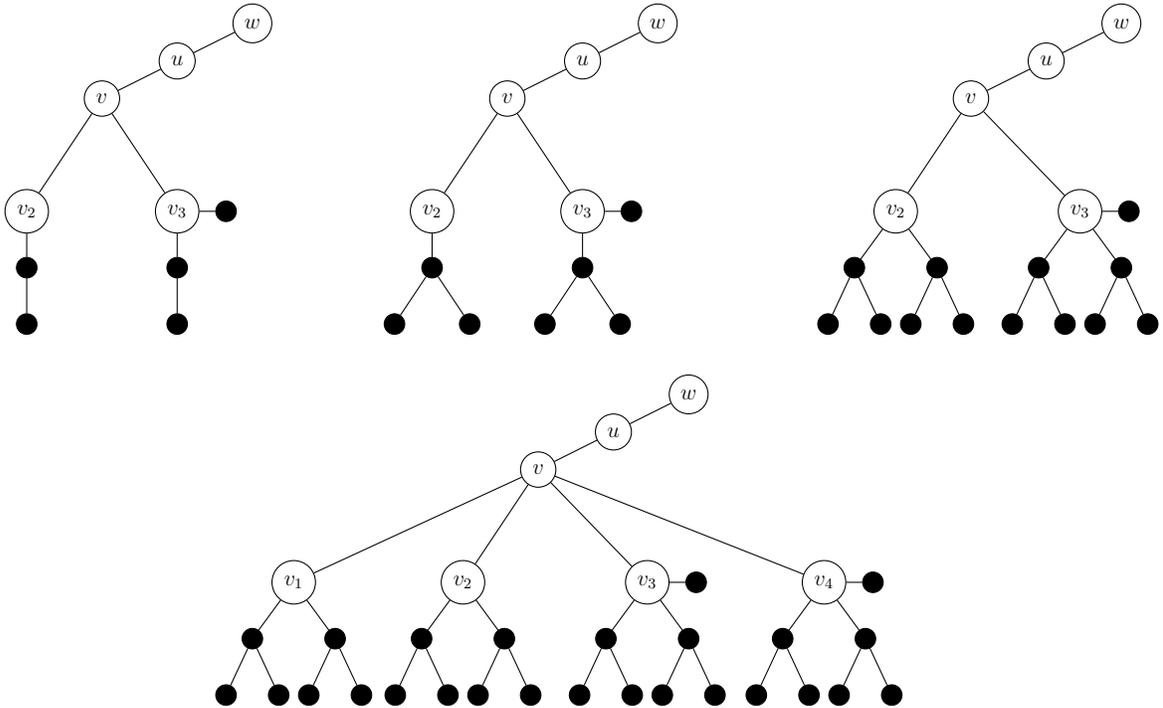}
    \caption{A sequence of CBD's starting with $S_7^{(9)}$ that achieves a tree $T$ (at the bottom) that is not diminimal.}
    \label{counter_example_2}
\end{figure}

  \begin{theorem}\label{badseedS69}
      The seed $S_7^{(9)}$ is a defective seed.
  \end{theorem}

  \begin{proof}
Consider the tree $T$ of diameter $7$ depicted in Figure~\ref{counter_example_2}. It is rooted at its central vertex $v$, which is unique because $T$ has odd diameter. Let $w$ be the leaf of $T$ for which $d(v,w)=2$ and let $u$ be the neighbor of $w$. Let $v_i$, for $i\in [4]=\{1,2,3,4\}$, be the children of $v$ that are not $u$, as they appear in Figure~\ref{counter_example_2}. Let $T_i$ be the connected component of $T-v$ that contains $v_i$, viewed as a rooted tree with root $v_i$, for $i\in [4]$. Let $\overline{T}=T_1\cup T_2\cup T_3\cup T_4$, so that $V(\overline{T})=V-\{u,v,w\}$.

Let $M\in\mathcal{S}(T)$. As in the previous proof, we write $\sigma(M)=\Spec(M)$. For any $\lambda\in\sigma(M)$, recall that $L(M,\lambda)$ is the distance between the main root and the closest vertex to the main root for which \texttt{Diagonalize}$(M,-\lambda)$ assigns value $0$. We start with the analogue of Claim~\ref{claim1} for this tree.

\begin{claim}\label{claim12}
The following are true:
\begin{enumerate}
    \item [(i)] If $\lambda \in \sigma(M)$ does not lie in $\sigma(M[\overline{T}])$, then $m_M(\lambda)=1$;
    \item [(ii)]         If $\lambda\in\sigma_j$ for $j>0$, then $\lambda\in\sigma                           (M)$ with $m_M(\lambda)=m_{M[\overline{T}]}(\lambda)$ or $m_M(\lambda)=m_{M[\overline{T}]}(\lambda)+1$;
    \item [(iii)] If $\lambda\in\sigma_0$, then $m_M(\lambda)=m_{M[\overline{T}]}(\lambda)$ if $\lambda\in \sigma(M[\{u,w\}])$ and $m_M(\lambda)=m_{M[\overline{T}]}(\lambda)-1$ otherwise.
\end{enumerate}
\end{claim}


\begin{proof}[Proof of Claim]
Consider an application of $\texttt{Diagonalize}(M,-\lambda)$ using an ordering for which $w$, $u$ and $v$ are the last three vertices. Before processing these three vertices, each remaining vertex $y$ is assigned the same value $d_y$ that would be assigned by $\texttt{Diagonalize}(M[T_i],-\lambda)$ to the corresponding component $T_i$.

For item $(i)$, suppose that $\lambda \notin \sigma(M[\overline{T}])$, all values are nonzero. By Theorem~\ref{thm_localizacao}(c) and given that $\lambda \in \sigma(M)$, we have $1 \leq m(\lambda) \leq 3$. Moreover, $m_M(\lambda)\geq 2$ if and only if two values among $d_u$, $d_v$ and $d_w$ are zero at the end of the algorithm. This is not possible. If $d_w=0$ upon processing $w$, then when processing $u$ the algorithm would assign $d_w>0$ and  $d_u<0$. If $d_w\neq 0$ and $d_u=0$ upon processing $u$, then when processing $v$ we would get $d_u>0$ and $d_v<0$.

In item $(ii)$, given that $\lambda\in \sigma_j$ with $j>0$, there are vertices $y \in V(\overline{T})$ for which $d_y=0$, but these vertices are not $v_1,\ldots,v_4$. As a consequence, the values assigned to these vertices will not change while processing $w$, $u$ and $v$. As explained in the previous paragraph, the algorithm may not assign value 0 for two quantities among $d_w$, $d_u$ and $d_v$. By Theorem~\ref{thm_localizacao}(c), $m_T(\lambda) \in \{m_{\overline{T}}(\lambda),m_{\overline{T}}(\lambda)+1\}$. 

For item $(iii)$ we suppose that $\lambda\in \sigma_0$. So, there are vertices $y \in V(\overline{T})$ for which $d_y=0$, at least one of which is $v_1,\ldots,v_4$. Next we process the leaf $w$, which has value $d_w=m_{ww}-\lambda$. Now, notice that the assignment of $w$ and $u$ will be the same as in $\texttt{Diagonalize}(M[\{u,w\}],-\lambda)$.

If $\lambda\in \sigma(M[\{u,w\}])$, then $d_u$ must be assigned with $0$ by $\texttt{Diagonalize}(M[\{u,w\}],-\lambda)$, because if $d_w=0$ upon processing $w$, then the algorithm would assign $d_u<0$ and $d_w>0$, a contradiction. While processing $v$, the value of one of its children (say $d_u$) is redefined as 2, and $d_v$ is assigned a negative number. Thus no new zero is produced, and $m_M(\lambda)=m_{M[\overline{T}]}(\lambda)$ by Theorem~\ref{thm_localizacao}(c).

If $\lambda \notin \sigma(M[\{u,w\}])$, we have $d_w,d_u\neq 0$. While processing $v$, the value of one of the $v_i$ for which $d_{v_i}=0$) is redefined as 2, and $d_v$ is assigned a negative number. This means that the number of zero values decreases by 1 and $m_M(\lambda)=m_{M[\overline{T}]}(\lambda)-1$ by Theorem~\ref{thm_localizacao}(c). 
\end{proof}

Coming back to our proof, let $\lambda_{\min}$ and $\lambda_{\max}$ denote the minimum and the maximum eigenvalue of $M$, respectively. By Theorem~\ref{thm:simpleroots}, we know that their multiplicity in $M$ is 1 and that $\lambda_{\min},\lambda_{\max}\notin\sigma(M[\overline{T}])$. 

Since $T_i$ has diameter $5$ for each $i \in [4]$, we know that $M[T_i]$ has at least five distinct eigenvalues. Moreover, by Lemma~\ref{no_zero}, there exist distinct real numbers $\mu_1,\mu_2,\mu_3$ such that $L(M[T_i],\mu_\ell)=0$ for all $\ell\in [3]$. Recall that any eigenvalue of $M[T_i]$ that is not an eigenvalue of $M$ is called a missing eigenvalue of branch $T_i$. 

To reach a contradiction, suppose that $|\DSpec(M)|=7$. Let the distinct eigenvalues of $M$ be $\lambda_{\min}<\lambda_1<\lambda_2<\lambda_3<\lambda_4<\lambda_5<\lambda_{\max}$. Since $\lambda_{\min},\lambda_{\max}\notin\sigma(M[\overline{T}])$, there are at most five distinct values that are in $\sigma_j$, for some $j\geq 0$ and in $\sigma(M)$. It is important to observe that $|V(\overline{T})|+1=n-2=m_{M}(\lambda_1)+m_{M}(\lambda_2)+m_{M}(\lambda_3)+m_{M}(\lambda_4)+m_{M}(\lambda_5)$. If each branch $T_i$ has $s_i$ missing eigenvalues, we get
\begin{equation}\label{eq302}
\sum_{\lambda\in\sigma(M[\overline{T}])\cap\sigma(M)}m_{M[\overline{T}]}(\lambda)=|V(\overline{T})|-(s_1+\cdots+s_4).
\end{equation}
Let $B=\{\lambda_1,\lambda_2,\lambda_3,\lambda_4,\lambda_5\}\setminus\sigma(M[\overline{T}])$ and $j=|B|$. We observe that any $\lambda_\ell \in B$ must be a simple eigenvalue of $M$ by Claim~\ref{claim12}(i).

Note that
\begin{eqnarray}\label{eq312}    
        |V(\overline{T})|+1 
        &=& m_{M}(\lambda_1)+m_{M}(\lambda_2)+m_{M}(\lambda_3)+m_{M}(\lambda_4)+m_{M}(\lambda_5)\nonumber\\\nonumber
        &=& \sum_{i=1}^j m_{M}(\lambda_i) + \sum_{\lambda\in\sigma_0\cap\sigma(M)} m_{M} (\lambda) + \sum_{\lambda\in\sigma_{\geq1}\cap\sigma(M)} m_{M} (\lambda)\\
        &=& j + \sum_{\lambda\in\sigma_0\cap\sigma(M)} m_{M} (\lambda) + \sum_{\lambda\in\sigma_{\geq1}\cap\sigma(M)} m_{M} (\lambda)
\end{eqnarray}
Claim~\ref{claim12}(ii) implies that
$$\sum_{\lambda\in\sigma_{\geq1}\cap\sigma(M)} m_{M} (\lambda) \leq \sum_{\lambda\in\sigma_{\geq1}} [m_{M[\overline{T}]} (\lambda)+1].$$
By Claim~\ref{claim12}(iii), we have
$$\sum_{\lambda\in\sigma_0\cap\sigma(M)} m_{M} (\lambda) \leq \delta_{u,w} + \sum_{\lambda\in\sigma_0\cap\sigma(M)} (m_{M[\overline{T}]}(\lambda)-1),$$
where $\delta_{u,w}=|\sigma_0\cap\sigma(M[\{u,w\}])|$. Note that $\delta_{u,w}\leq 2$, because $M[\{u,w\}]$ has two vertices, so $|\sigma(M[\{u,w\}])|\leq 2$. As a consequence, \eqref{eq312} implies
\begin{eqnarray}
        |V(\overline{T})|&+&1 \leq  j+\delta_{u,w} + \sum_{\lambda\in\sigma(M[\overline{T}])\cap\sigma(M)}m_{M[\overline{T}]}(\lambda) - |\sigma_0\cap\sigma(M)| +|\sigma_{\geq1}\cap\sigma(M)| \nonumber \\
        &\leq& j+\delta_{u,w} + |V(\overline{T})| - (s_1+\cdots+s_4) - |\sigma_0\cap\sigma(M)| +|\sigma_{\geq1}\cap\sigma(M)|. \label{eq322}
        \end{eqnarray}

Summing $2|\sigma_{0}\cap\sigma(M)|$ to both sides of the inequality \eqref{eq322}, and performing a few algebraic manipulations, we get
\begin{equation}\label{eq3321}
|\sigma_{0}\cap\sigma(M)|+|\sigma_{\geq1}\cap\sigma(M)| \geq 2|\sigma_0\cap\sigma(M)| + (s_1+\cdots+s_4)-j-\delta_{u,w}+1.
\end{equation}  
Moreover, given that $\delta_{u,w}\leq 2$, we have 
\begin{equation}\label{eq332}
|\sigma_{0}\cap\sigma(M)|+|\sigma_{\geq1}\cap\sigma(M)| \geq 2|\sigma_0\cap\sigma(M)| + (s_1+\cdots+s_4)-j-1.
\end{equation}

Equation \eqref{eq332} is exactly the same as \eqref{eq33}, which appears in the proof of Theorem~\ref{badseedS68}. Using the same reasoning, we conclude that $\sigma(M[\overline{T}])=\{\lambda_1,\ldots,\lambda_5\} \subset \sigma(M)$ and that $|\sigma_0 \cap \sigma(M)|=3$. However, the proof of Theorem~\ref{badseedS68} establishes that it is impossible to find real numbers $\lambda_1<\cdots<\lambda_5$ with this property. Thus there is no matrix $M\in \mathcal{S}(T)$ such that $|\DSpec(M)|=7$.
\end{proof}

\newcommand{\counterexampletres}{
\begin{tikzpicture}[scale=1,auto=left,every node/.style={circle,scale=0.7}]

\path(0,1)node[shape=circle,draw,fill=white] (a) {$v$}
     (1,1.5)node[shape=circle,draw,fill=white] (b) {$u$}
     (2,2)node[shape=circle,draw,fill=white] (c) {$w$}
     (-1,1.5)node[shape=circle,draw,fill=white] (d) {$x$}
     
     (-1   ,-0.5)node[shape=circle,draw,fill=white] (0) {$v_2$}
     (-1.55,-1.25)node[shape=circle,draw,fill=black] (01) {}
     (-0.45,-1.25)node[shape=circle,draw,fill=black] (02) {}
     (-1.2 ,-2)node[shape=circle,draw,fill=black]    (011) {}
     (-1.9   ,-2)node[shape=circle,draw,fill=black]  (012) {}
     (-0.1    ,-2)node[shape=circle,draw,fill=black] (021) {}
     (-0.8 ,-2)node[shape=circle,draw,fill=black]    (022) {}

     (-3.25   ,-0.5)node[shape=circle,draw,fill=white] (1) {$v_1$}
     (-3.8,-1.25)node[shape=circle,draw,fill=black] (11) {}
     (-2.7,-1.25)node[shape=circle,draw,fill=black] (12) {}
     (-3.45 ,-2)node[shape=circle,draw,fill=black]    (111) {}
     (-4.15   ,-2)node[shape=circle,draw,fill=black]  (112) {}
     (-2.35    ,-2)node[shape=circle,draw,fill=black] (121) {}
     (-3.05 ,-2)node[shape=circle,draw,fill=black]    (122) {}

     (-1   +2.45,-0.5)node[shape=circle,draw,fill=white] (2) {$v_3$}
     (-0.35+2.45,-0.5)node[shape=circle,draw,fill=black]  (20) {}
     (-1.55+2.45,-1.25)node[shape=circle,draw,fill=black] (21) {}
     (-0.45+2.45,-1.25)node[shape=circle,draw,fill=black] (22) {}
     (-1.2 +2.45,-2)node[shape=circle,draw,fill=black]    (211) {}
     (-1.9 +2.45  ,-2)node[shape=circle,draw,fill=black]  (212) {}
     (-0.1 +2.45   ,-2)node[shape=circle,draw,fill=black] (221) {}
     (-0.8 +2.45,-2)node[shape=circle,draw,fill=black]    (222) {}

     (-1   +4.8,-0.5)node[shape=circle,draw,fill=white] (3) {$v_4$}
     (-0.35+4.8,-0.5)node[shape=circle,draw,fill=black]  (30) {}
     (-1.55+4.8,-1.25)node[shape=circle,draw,fill=black] (31) {}
     (-0.45+4.8,-1.25)node[shape=circle,draw,fill=black] (32) {}
     (-1.2 +4.8,-2)node[shape=circle,draw,fill=black]    (311) {}
     (-1.9 +4.8  ,-2)node[shape=circle,draw,fill=black]  (312) {}
     (-0.1 +4.8   ,-2)node[shape=circle,draw,fill=black] (321) {}
     (-0.8 +4.8,-2)node[shape=circle,draw,fill=black]    (322) {} ;

      \draw[-](a)--(b);
      \draw[-](a)--(d);
      \draw[-](b)--(c);
      \draw[-](a)--(0);
      \draw[-](a)--(1);
      \draw[-](a)--(2);
      \draw[-](a)--(3);

      \draw[-](0)--(01);
      \draw[-](0)--(02);
      \draw[-](01)--(011);
      \draw[-](01)--(012);
      \draw[-](02)--(021);
      \draw[-](02)--(022);

      \draw[-](1)-- (11);
      \draw[-](1)-- (12);
      \draw[-](11)--(111);
      \draw[-](11)--(112);
      \draw[-](12)--(121);
      \draw[-](12)--(122);

      \draw[-](2)-- (20);
      \draw[-](2)-- (21);
      \draw[-](2)-- (22);
      \draw[-](21)--(211);
      \draw[-](21)--(212);
      \draw[-](22)--(221);
      \draw[-](22)--(222);

      \draw[-](3)-- (30);
      \draw[-](3)-- (31);
      \draw[-](3)-- (32);
      \draw[-](31)--(311);
      \draw[-](31)--(312);
      \draw[-](32)--(321);
      \draw[-](32)--(322);
    
     \end{tikzpicture}
}

\begin{figure}[H]
    \centering
    \hspace{-1cm}
    \input{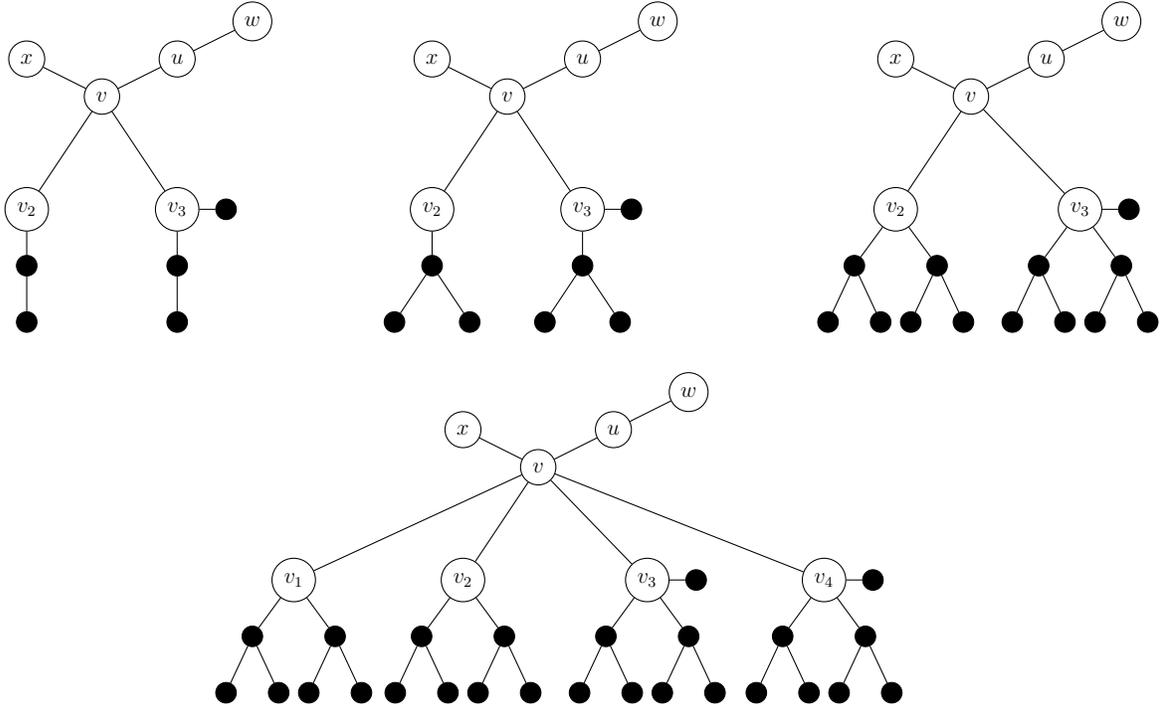}
    \caption{A sequence of CBD's starting with $S_7^{(7)}$ that achieves a tree $T$ (at the bottom) that is not diminimal.}
    \label{counter_example_3}
\end{figure}


Finally we prove that $S_7^{(7)}$ is a defective seed. The proof is also similar to the proof of Theorem~\ref{badseedS68}.
  \begin{theorem}
      The seed $S_7^{(7)}$ is a defective seed.
  \end{theorem}

\begin{proof}
Consider the tree $T$ of diameter $7$ depicted in Figure~\ref{counter_example_3}. It is rooted at vertex $v$. Let $w$ be the leaf of $T$ for which $d(v,w)=2$ and let $u$ be the neighbor of $w$. Let $x$ be the leaf connected to $v$. Let $v_i$, for $i\in [4]=\{1,2,3,4\}$, be the children of $v$ that are not $u$ nor $x$, as they appear in Figure~\ref{counter_example_3}. Let $T_i$ be the connected component of $T-v$ that contains $v_i$, viewed as a rooted tree with root $v_i$, for $i\in [4]$. Let $\overline{T}=T_1\cup T_2\cup T_3\cup T_4$ and let $T_0=T[v,u,w,x]$, so that $V(\overline{T})=V-T_0$.

Let $M\in\mathcal{S}(T)$. As before, we write $\sigma(M)=\Spec(M)$, and, for any $\lambda\in\sigma(M)$, $L(M,\lambda)$ denotes the distance between the root and the closest vertex to the root assigned with $0$. We again start with a claim about the eigenvalues of $M$.

\begin{claim}\label{claim13}
The following are true:
\begin{enumerate}
    \item [(i)] If $\lambda \in \sigma(M)$ does not lie in $\sigma(M[\overline{T}])$, then $m_M(\lambda)=1$;
    \item [(ii)] If $\lambda\in\sigma_j$ for $j>0$, then $\lambda\in\sigma                           (M)$ with $m_M(\lambda)=m_{M[\overline{T}]}(\lambda)$ or $m_M(\lambda)=m_{M[\overline{T}]}(\lambda)+1$;
    \item [(iii)] If $\lambda\in\sigma_0$, one of the following cases is true:
$$\displaystyle{m_M(\lambda) =
\begin{cases}
m_{M[\overline{T}]}(\lambda)+1, & \textrm{ if } m_{xx}=\lambda \textrm{ and }\lambda\in \sigma(M[\{u,w\}]) ,\\
m_{M[\overline{T}]}(\lambda), & \textrm{ if either } m_{xx}=\lambda \textrm{ or } \lambda\in \sigma(M[\{u,w\}]),\\
m_{M[\overline{T}]}(\lambda)-1, & \textrm{ otherwise}.
\end{cases}}$$
\end{enumerate}
\end{claim}


\begin{proof}[Proof of Claim]
Consider an application of $\texttt{Diagonalize}(M,-\lambda)$ using an ordering for which the vertices in $T_0$ are processed in the end. Before processing $T_0$, each remaining vertex $y$ is assigned the same value $d_y$ that would be assigned by $\texttt{Diagonalize}(M[T_i],-\lambda)$ to the corresponding component $T_i$. 

For item $(i)$, suppose that $\lambda \notin \sigma(M[\overline{T}])$, so that all values assigned by the algorithm are nonzero. By Theorem~\ref{thm_localizacao}(c) and given that $\lambda \in \sigma(M)$, we have $1 \leq m(\lambda)$. Moreover, $m_M(\lambda)\geq 2$ if and only if the value assigned to at least two out of the four vertices of $T_0$ are zero at the end of the algorithm. This is not possible. If $d_w=0$ upon processing $w$, then when processing $u$ we would have $d_w>0$ and $d_u<0$; in this case, if $d_x=0$ upon processing $x$, then when processing $v$ we would have $d_x>0$ and $d_v<0$. 

So assume that $d_w\neq0$. If upon processing $u$, we have $d_u=0$, when processing $v$ we would get $d_v<0$ and $d_u>0$, in addition to $d_w\neq 0$. Next, assume $d_w,d_u\neq0$ upon processing $u$. If we have $d_x=0$ upon processing $x$, when processing $v$ we would have $d_v<0$ and $d_x>0$, in addition to $d_w,d_u\neq 0$. So there is no case where two or more vertices $y$ of $T_0$ are assigned value $d_y=0$.

Regarding $(ii)$, given that $\lambda\in \sigma_j$ with $j>0$, there are vertices $y \in V(\overline{T})$ for which $d_y=0$, but these vertices are not $v_1,\ldots,v_4$. As a consequence, the values assigned to these vertices will not change while processing $T_0$. As explained in the previous paragraph, the algorithm may not assign value 0 for two or more vertices of $T_0$. By Theorem~\ref{thm_localizacao}(c), $m_T(\lambda) \in \{m_{\overline{T}}(\lambda),m_{\overline{T}}(\lambda)+1\}$. 

To prove $(iii)$ we suppose that $\lambda\in \sigma_0$. Then, there are vertices $y \in V(\overline{T})$ for which $d_y=0$, at least one of which lies in $\{v_1,\ldots,v_4\}$. Next we process the leaf $w$, which has value $d_w=m_{ww}-\lambda$. Now, notice that the assignment of $w$ and $u$ will be the same as in $\texttt{Diagonalize}(M[\{u,w\}],-\lambda)$. Next, we process the leaf $x$ which will have assignment $d_x=m_{ww}-\lambda$.

First, consider the case where $\lambda\in \sigma(M[\{u,w\}])$ and $m_{xx}=\lambda$. Because $\lambda\in \sigma(M[\{u,w\}])$, $d_u$ must be assigned with $0$ by $\texttt{Diagonalize}(M[\{u,w\}],-\lambda)$, because if $d_w=0$, then the algorithm would assign $d_u<0$ and $d_w>0$. Because $m_{xx}=\lambda$, we have $d_x=0$. While processing $v$, the value of one of its children (say $d_u$) is redefined as 2, and $d_v$ is assigned a negative number. Thus one new zero is produced (for vertex $x$), and $m_M(\lambda)=m_{M[\overline{T}]}(\lambda)+1$ by Theorem~\ref{thm_localizacao}(c).

Secondly, consider the case where one, and only one, of the following cases is true: 
\begin{itemize}
    \item [(a)] $\lambda\in \sigma(M[\{u,w\}])$;
    \item [(b)]$m_{xx}=\lambda$.
\end{itemize}
If (a) is true, then, as explained above, we will have $d_u=0$ and $d_w\neq 0$. Since (b) is not true, we have $d_x=m_{xx}-\lambda\neq0$. While processing $v$, the value of one of its children (say $d_u$) is redefined as 2, and $d_v$ is assigned a negative number. Thus no new zero is produced, and $m_M(\lambda)=m_{M[\overline{T}]}(\lambda)$ by Theorem~\ref{thm_localizacao}(c).

Now, if (b) is true, we will have $d_x=0$. Since (a) is not true, we have $d_w,d_u\neq 0$. While processing $v$, the value of one of its children (say $d_x$) is redefined as 2, and $d_v$ is assigned a negative number. Thus no new zero is produced, and $m_M(\lambda)=m_{M[\overline{T}]}(\lambda)$ by Theorem~\ref{thm_localizacao}(c).

Finally, suppose $\lambda\notin \sigma(M[\{u,w\}])$ and $m_{xx}\neq\lambda$. Since $\lambda\notin \sigma(M[\{u,w\}])$, then $d_w,d_u\neq 0$. And, since $m_{xx}\neq\lambda$ we have $d_x=m_{xx}-\lambda\neq0$. While processing $v$, the value of one of the $v_i$ for which $d_{v_i}=0$) is redefined as 2, and $d_v$ is assigned a negative number. This means that the number of zero values decreases by 1 and $m_M(\lambda)=m_{M[\overline{T}]}(\lambda)-1$ by Theorem~\ref{thm_localizacao}(c). 
\end{proof}

Coming back to our proof, we follow the steps of the proof of Theorem~\ref{badseedS68}. In particular, proof $\lambda_{\min}$ and $\lambda_{\max}$ are such that their multiplicity in $M$ is 1 and that $\lambda_{\min},\lambda_{\max}\notin\sigma(M[\overline{T}])$. 

Each $M[T_i]$ has at least five distinct eigenvalues. Moreover, by Lemma~\ref{no_zero}, there exist distinct real numbers $\mu_1,\mu_2,\mu_3$ such that $L(M[T_i],\mu_\ell)=0$ for all $\ell\in [3]$. Recall that an eigenvalue of $M[T_i]$ that is not an eigenvalue of $M$ is a missing eigenvalue of branch $T_i$. 
        
To reach a contradiction, suppose that $|\DSpec(M)|=7$. Let the distinct eigenvalues of $M$ be $\lambda_{\min}<\lambda_1<\lambda_2<\lambda_3<\lambda_4<\lambda_5<\lambda_{\max}$. Since $\lambda_{\min},\lambda_{\max}\notin\sigma(M[\overline{T}])$, there are at most five distinct values that are in $\sigma_j$, for some $j\geq 0$ and in $\sigma(M)$. It is important to observe that $|V(\overline{T})|+2=n-2=m_{M}(\lambda_1)+m_{M}(\lambda_2)+m_{M}(\lambda_3)+m_{M}(\lambda_4)+m_{M}(\lambda_5)$. If each branch $T_i$ has $s_i$ missing eigenvalues, we get
\begin{equation}\label{eq303}
\sum_{\lambda\in\sigma(M[\overline{T}])\cap\sigma(M)}m_{M[\overline{T}]}(\lambda)=|V(\overline{T})|-(s_1+\cdots+s_4).
\end{equation}
Let $B=\{\lambda_1,\lambda_2,\lambda_3,\lambda_4,\lambda_5\}\setminus\sigma(M[\overline{T}])$ and $j=|B|$. We observe that any $\lambda_\ell \in B$ must be a simple eigenvalue of $M$ by Claim~\ref{claim13}(i).

Note that
\begin{eqnarray}\label{eq313}    
        |V(\overline{T})|+2
        &=& m_{M}(\lambda_1)+m_{M}(\lambda_2)+m_{M}(\lambda_3)+m_{M}(\lambda_4)+m_{M}(\lambda_5)\nonumber\\\nonumber
        &=& \sum_{i=1}^j m_{M}(\lambda_i) + \sum_{\lambda\in\sigma_0\cap\sigma(M)} m_{M} (\lambda) + \sum_{\lambda\in\sigma_{\geq1}\cap\sigma(M)} m_{M} (\lambda)\\
        &=& j + \sum_{\lambda\in\sigma_0\cap\sigma(M)} m_{M} (\lambda) + \sum_{\lambda\in\sigma_{\geq1}\cap\sigma(M)} m_{M} (\lambda)
\end{eqnarray}
Claim~\ref{claim13}(ii) implies that
$$\sum_{\lambda\in\sigma_{\geq1}\cap\sigma(M)} m_{M} (\lambda) \leq \sum_{\lambda\in\sigma_{\geq1}} [m_{M[\overline{T}]} (\lambda)+1].$$
By Claim~\ref{claim13}(iii), we have
$$\sum_{\lambda\in\sigma_0\cap\sigma(M)} m_{M} (\lambda) \leq \delta_{x}+\delta_{u,w} + \sum_{\lambda\in\sigma_0\cap\sigma(M)} (m_{M[\overline{T}]}(\lambda)-1),$$
where $\delta_{u,w}=|\sigma_0\cap\sigma(M[\{u,w\}])|$ and $\delta_{x}=1$ if $m_{xx}$ lies in the set $\{\lambda \colon \lambda \in \sigma_0\cap\sigma(M)\}$ and $\delta_{x}=0$ otherwise. Note that $\delta_{u,w}\leq 2$, because $M[\{u,w\}]$ has two vertices, then $|\sigma(M[\{u,w\}])|\leq 2$. Clearly, $\delta_{x}\leq 1$.

As a consequence, \eqref{eq313} leads to
\begin{eqnarray}
        &&|V(\overline{T})|+2 \leq j+\delta_{x}+\delta_{u,w} + \sum_{\lambda\in\sigma(M[\overline{T}])\cap\sigma(M)}m_{M[\overline{T}]}(\lambda) - |\sigma_0\cap\sigma(M)| +|\sigma_{\geq1}\cap\sigma(M)| \nonumber \\
        &\leq& j+\delta_{x}+\delta_{u,w} + |V(\overline{T})| - (s_1+\cdots+s_4) - |\sigma_0\cap\sigma(M)| +|\sigma_{\geq1}\cap\sigma(M)|. \label{eq323}
        \end{eqnarray}

Using the inequality \eqref{eq323} and summing $2|\sigma_{0}\cap\sigma(M)|$ on both sides, we get
\begin{equation}\label{eq333}
|\sigma_{0}\cap\sigma(M)|+|\sigma_{\geq1}\cap\sigma(M)| \geq 2|\sigma_0\cap\sigma(M)| + (s_1+\cdots+s_4)-j-\delta_{x}-\delta_{u,w}+2.
\end{equation}  
Moreover, given that $\delta_{u,w}\leq 2$ and $\delta_{x}\leq 1$, we have
\begin{equation}\label{eq342}
|\sigma_{0}\cap\sigma(M)|+|\sigma_{\geq1}\cap\sigma(M)| \geq 2|\sigma_0\cap\sigma(M)| + (s_1+\cdots+s_4)-j-1.
\end{equation}

Equation \eqref{eq342} is exactly the same as \eqref{eq33}, which appears in the proof of Theorem~\ref{badseedS68}. Using the same reasoning, we conclude that $\sigma(M[\overline{T}])=\{\lambda_1,\ldots,\lambda_5\} \subset \sigma(M)$ and that $|\sigma_0 \cap \sigma(M)|=3$. However, the proof of Theorem~\ref{badseedS68} establishes that it is impossible to find real numbers $\lambda_1<\cdots<\lambda_5$ with this property. Thus there is no matrix $M\in \mathcal{S}(T)$ such that $|\DSpec(M)|=7$.

\end{proof}

\section{Proof of Theorem~\ref{t2}}\label{sec:thm2}

The aim of this section is to prove that, for any seed 
$S_7^{(i)}$ such that $7\leq i \leq 9$ and any tree $T \in \mathcal{T}(S_7^{(i)})$, there exists a matrix $M$ with underlying tree $T$ such that $|\DSpec(M)|=8$.

We start with some notation. Let $v$ denote the central vertex of each $S_7^{(7)}$, $S_7^{(8)}$ and $S_7^{(9)}$, that is, the midpoint of any maximum path in the seed. Let $v_1,v_2$ be the neighbors of $v$ that are not leaves and that have at least two non-leaf neighbors (including $v$). Vertex $v_2$ is adjacent to a leaf. Figure~\ref{bad_seeds} depicts the three seeds. Any diameter-preserving unfoldings of these seeds are either CBDs of branches at $v$ or CBDs at other vertices for branches that do not contain 
$v$.

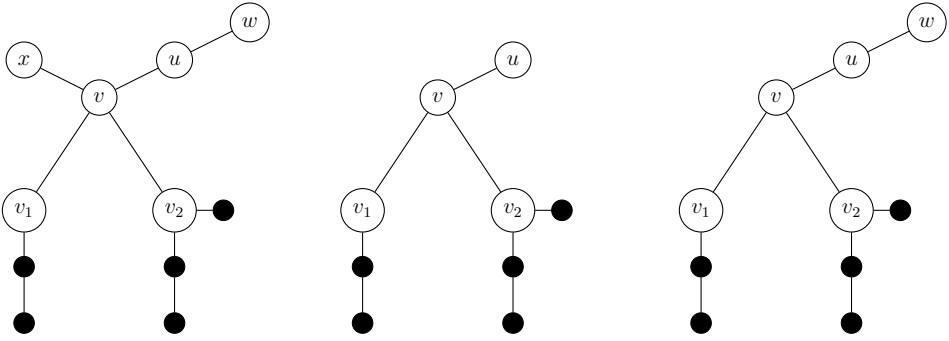
\begin{figure}
    \centering
    \hspace{-1cm}
    \begin{tikzpicture}[scale=1,auto=left,every node/.style={circle,scale=0.7}]

\path(0,1)node[shape=circle,draw,fill=white] (a) {$v$}
     (1,1.5)node[shape=circle,draw,fill=white] (b) {$u$}
     (2,2)node[shape=circle,draw,fill=white] (c) {$w$}
     (-1,1.5)node[shape=circle,draw,fill=white] (d) {$x$} 
     (-1   ,-0.5)node[shape=circle,draw,fill=white] (0) {$v_1$}
     (-1,-1.25)node[shape=circle,draw,fill=black] (01) {}
     (-1 ,-2)node[shape=circle,draw,fill=black]    (011) {}
     (-1   +2,-0.5)node[shape=circle,draw,fill=white] (2) {$v_2$}
     (-0.35+2,-0.5)node[shape=circle,draw,fill=black]  (20) {}
     (-1+2,-1.25)node[shape=circle,draw,fill=black] (21) {}
     (-1 +2,-2)node[shape=circle,draw,fill=black]    (211) {}

     (0+4.5,1)node[shape=circle,draw,fill=white] (a8) {$v$}
     (1+4.5,1.5)node[shape=circle,draw,fill=white] (b8) {$u$}
     (-1 +4.5  ,-0.5)node[shape=circle,draw,fill=white] (08) {$v_1$}
     (-1+4.5,-1.25)node[shape=circle,draw,fill=black] (018) {}
     (-1 +4.5,-2)node[shape=circle,draw,fill=black]    (0118) {}
     (-1   +2+4.5,-0.5)node[shape=circle,draw,fill=white] (28) {$v_2$}
     (-0.35+2+4.5,-0.5)node[shape=circle,draw,fill=black]  (208) {}
     (-1+2+4.5,-1.25)node[shape=circle,draw,fill=black] (218) {}
     (-1 +2+4.5,-2)node[shape=circle,draw,fill=black]    (2118) {}

     (0+9,1)node[shape=circle,draw,fill=white] (a9) {$v$}
     (1+9,1.5)node[shape=circle,draw,fill=white] (b9) {$u$}
     (2+9,2)node[shape=circle,draw,fill=white] (c9) {$w$}
     (-1   +9,-0.5)node[shape=circle,draw,fill=white] (09) {$v_1$}
     (-1+9,-1.25)node[shape=circle,draw,fill=black] (019) {}
     (-1 +9,-2)node[shape=circle,draw,fill=black]    (0119) {}
     (-1   +2+9,-0.5)node[shape=circle,draw,fill=white] (29) {$v_2$}
     (-0.35+2+9,-0.5)node[shape=circle,draw,fill=black]  (209) {}
     (-1+2+9,-1.25)node[shape=circle,draw,fill=black] (219) {}
     (-1 +2+9,-2)node[shape=circle,draw,fill=black]    (2119) {}
     ;
     
      \draw[-](a)--(b);
      \draw[-](a)--(0);
      \draw[-](b)--(c);
      \draw[-](a)--(d);
      \draw[-](a)--(2);
      \draw[-](0)--(01);
      \draw[-](01)--(011);
      \draw[-](2)-- (20);
      \draw[-](2)-- (21);
      \draw[-](21)--(211);

      \draw[-](a8)--(b8);
      \draw[-](a8)--(08);
      \draw[-](a8)--(28);
      \draw[-](08)--(018);
      \draw[-](018)--(0118);
      \draw[-](28)-- (208);
      \draw[-](28)-- (218);
      \draw[-](218)--(2118);

      \draw[-](a9)--(b9);
      \draw[-](a9)--(09);
      \draw[-](b9)--(c9);
      \draw[-](a9)--(29);
      \draw[-](09)--(019);
      \draw[-](019)--(0119);
      \draw[-](29)-- (209);
      \draw[-](29)-- (219);
      \draw[-](219)--(2119);

     \end{tikzpicture}
    \caption{The seeds $S_7^{(7)}$, $S_7^{(8)}$ and $S_7^{(9)}$.}
    \label{bad_seeds}
\end{figure}

We split each seed into three parts as follows. Let $S \in \{S_7^{(7)},S_7^{(8)},S_7^{(9)}\}$.
\begin{enumerate}
    \item[(i)] 
    Let $S_0$ be the connected component of  $S\setminus\{v_1,v_2\}$ that contains $v$;
    \item[(ii)] Let $S_1$ be the connected component of  $S\setminus \{v\}$ that contains $v_1$;
    \item[(iii)] Let $S_2$ be the connected component of  $S\setminus \{v\}$ that contains $v_2$.
\end{enumerate}

Let $S \in \{S_7^{(7)},S_7^{(8)},S_7^{(9)}\}$. We consider a canonical way to perform the CDBs to produce any $T \in \mathcal{T}(S)$. Given $q_1,q_2\in\mathbb{N}$ we first do a $(q_1-1)$-CBD of the branch $S_1$ at $v$ and a $(q_2-1)$-CBD of the branch $S_2$ at $v$. After this, if we remove the vertices of $S_0$, we have components $S^{(1)}_1,\ldots,S^{(q_1)}_1$ isomorphic to $S_1$, each of which is viewed as rooted at the vertex connected to $v$, denoted by $v^{(i)}_1$, and components $S^{(1)}_2,\ldots,S^{(q_2)}_2$ isomorphic to $S_2$, each of which is rooted at the vertex connected to $v$, denoted by $v^{(j)}_2$. Next, we consider CBDs in each of these components. The unfolding of each $S^{(i)}_1$ will be denoted by $T^{(i)}_1$ and the submatrix associated with this branch will be $M(T^{(i)}_1)$. The notation $T^{(j)}_2$ and $M^{(j)}_2$ is defined analogously with respect to $S^{(j)}_2$. The unfolding of all branches in $S_0$ will be denoted by $T_0$, and the submatrix associated with it is $M(T_0)$. Figure~\ref{generalform} depicts the overall structure of the unfolding of $S$.

 \begin{figure}
\centering
\begin{tikzpicture}[scale=.7,auto=left,every node/.style={circle,scale=0.5}]

\path (-4,0)node[shape=circle,draw,fill=white] (02) {$M_1(T^{(1)}_{1})$}
      (-4,0.75)node[shape=rectangle,label=above:\Large{$v^{(1)}_1$},draw,fill=red,minimum size=0.5cm] (2) {}
      
      (-1,0)node[shape=circle,draw,fill=white] (01) {$M_1(T^{(q_1)}_{1})$}
      (-1,0.75)node[shape=rectangle,label=above:\Large{$v^{(q_1)}_1$},draw,fill=red,minimum size=0.5cm] (1) {}

      (4,0)node[shape=circle,draw,fill=white] (06) {$M_2(T^{(q_2)}_{2})$}
      (4,0.75)node[shape=rectangle,label=above:\Large{$v^{(q_2)}_2$},draw,fill=red,minimum size=0.5cm] (6) {}
      
      (1,0)node[shape=circle,draw,fill=white] (05) {$M_2(T^{(1)}_{2})$}
      (1,0.75)node[shape=rectangle,label=above:\Large{$v^{(1)}_2$},draw,fill=red,minimum size=0.5cm] (5) {}
      
      (0,3.5)node[shape=circle,draw,fill=white,minimum size=2cm] (03) {$M(T_0)$}
      (0,2.75)node[shape=rectangle,label=below:\Large{$v$},draw,fill=red,minimum size=0.5cm] (3) {};

     \draw[-](3)--(2);
     \draw[dotted](02)--(01);
     \draw[-](3)--(1);
     \draw[-](3)--(5);
     \draw[dotted](05)--(06);
     \draw[-](3)--(6);

\end{tikzpicture}     
\caption{General unfolding of a seed $S \in \{S_7^{(7)},S_7^{(8)},S_7^{(9)}\}$.}
\label{generalform}
 \end{figure}
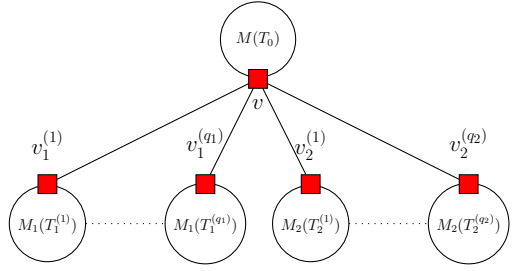

Furthermore we define the notation for general unfoldings of $S_1$ and $S_2$. Let $p$ and $t_0,t_1,\ldots,t_p$ be non-negative integers. For $S_1$, the tree $T_1(t_1,\ldots,t_p)$ is obtained by performing a $(p-1)$-CBD of the single branch of $S_1$ at $v_1$, and by performing $(t_i-1)$-CBDs of the leaves connected to each of the $p$ children of $v_1$. Similarly, for $S_2$, the tree $T_2(t_0,t_1,\ldots,t_p)$ is obtained by performing a $(p-1)$-CBD of the branch of $S_2$ at $v_2$ that is a 2-path, $(t_i-1)$-CBDs of the leaves connected to each of the $p$ children of $v_2$, and a $(t_0-1)$-CBD of the branch at $v_2$ that contains a single vertex.

Figure~\ref{q8T1} depicts $T_1(t_1,\ldots,t_p)$ and the entries that will be assigned to the submatrix of $M$ associated with the vertices of this tree. Diagonal entries are shown as vertex weights. Off-diagonal entries are shown as edge weights. If we consider the submatrices in Figure~\ref{generalform}, given $i\in\{1,\ldots,q_1\}$, let $p_{i,1}\geq 1$ and $t^{(i)}_1,\ldots,t^{(i)}_{p_{i,1}}\geq 1$. We have $T^{(i)}_1 = T_1(t^{(i)}_1,\ldots,t^{(i)}_{p_{i,1}})$. Moreover,  $M(T^{(i)}_1)\in\mathcal{S}(T^{(i)}_1)$ is the matrix given by the weights in Figure~\ref{q8T1}.
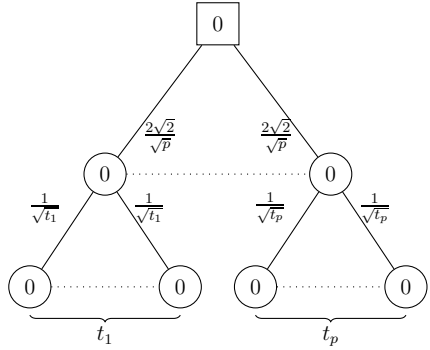
\begin{figure}
\centering
 \begin{tikzpicture}
[scale=1,auto=left,every node/.style={circle,scale=0.7}]

\path( 0.5,0)  node[shape=circle,draw,minimum size=0.8cm,label=above:,inner sep=0] (v1)  {$0$}
     ( 2.5,0)  node[shape=circle,draw,minimum size=0.8cm,label=above:,inner sep=0] (v2)  {$0$}
     ( 1.5,1.5)node[shape=circle,draw,minimum size=0.8cm,label=above:,inner sep=0] (v12) {$0$}

     ( 3.5,0)  node[shape=circle,draw,minimum size=0.8cm,label=above:,inner sep=0] (v3)  {$0$}
     ( 5.5,0)  node[shape=circle,draw,minimum size=0.8cm,label=above:,inner sep=0] (v4)  {$0$}
     ( 4.5,1.5)node[shape=circle,draw,minimum size=0.8cm,label=above:,inner sep=0] (v11) {$0$}

     ( 3,3.5)  node[shape=rectangle,draw,minimum size=0.8cm,label=above:,inner sep=0] (v13) {$0$};

\draw[dotted](v1) -- (v2) ; 
\draw[-](v1) edge node{$\frac{1}{\sqrt{t_1}}$} (v12) ; 
\draw[-,pos=0.375,above](v2) edge node{$\frac{1}{\sqrt{t_1}}$} (v12) ;

\draw[decoration={brace,mirror,raise=5pt},decorate]
  (0.5,-0.2) -- node[right=0pt,below,yshift=-0.5em] {$t_1$} (2.5,-0.2);
  
\draw[dotted](v3) -- (v4) ; 
\draw[-](v3) edge node{$\frac{1}{\sqrt{t_p}}$}  (v11) ; 
\draw[-,pos=0.375,above](v4) edge node{$\frac{1}{\sqrt{t_p}}$}  (v11) ; 

\draw[decoration={brace,mirror,raise=5pt},decorate]
  (3.5,-0.2) -- node[right=0pt,below,yshift=-0.5em] {$t_p$} (5.5,-0.2);

\draw[dotted](v12) -- (v11) ; 
\draw[-,below](v13) edge node{$\frac{2\sqrt{2}}{\sqrt{p}}$}  (v12) ; 
\draw[-,below](v13) edge node{$\frac{2\sqrt{2}}{\sqrt{p}}$}  (v11) ; 

\end{tikzpicture}
            \caption{A general unfolding $T_1 = T_1(t_1,\ldots,t_p)$ of $S_1$, rooted at the vertex $v_1$, which depicted as a square. Vertex and edge-weights denote the entries of the matrix associated with it.} 
            \label{q8T1}
\end{figure}
Lemma~\ref{spectraT1} gives information about the spectrum of this matrix, which is a subset of $\{-3,-1,0,1,3\}$
regardless of the choice of the parameters $t_1,\ldots,t_p$. 
We refer to Appendix~\ref{appen} for a proof of this lemma.

\begin{lemat}\label{spectraT1}
    If $M=M(T_1)$ is the matrix in Figure~\ref{q8T1} for the tree $T_1=T_1(t_1,\ldots,t_p)$, then $\DSpec(M)\subseteq\{-3,-1,0,1,3\}$. Moreover, $L(M,\lambda)=0$ for $\lambda\in\{-3,0,3\}$ and $m_{M[T_1\setminus v]}(\lambda) = m_{M}(\lambda)+1$ for $\lambda\in\{-1,1\}$. Also, \texttt{Diagonalize}$(M,-2)$ assigns $d_{v_1}=\frac{10}{3}$ to the root $v_1$ of $T_1$.
\end{lemat}

Analogously, Figure~\ref{q8T2} depicts $T_2(t_0,\ldots,t_p)$, as well as the entries that will be assigned to the submatrix of $M$ associated with the vertices of this tree. As before, diagonal entries are shown as vertex weights and off-diagonal entries are shown as edge weights. Turning back to Figure~\ref{generalform}, given $j\in\{1,\ldots,q_2\}$ let $p_{j,2}\geq 1$ and $t^{(j)}_1,\ldots,t^{(j)}_{p_{j,2}}\geq 1$, then we have $T^{(j)}_2 = T_2(t^{(j)}_0,\ldots,t^{(j)}_{p_{j,2}})$. Moreover,  $M(T^{(j)}_2)\in\mathcal{S}(T^{(j)}_2)$ is the matrix in Figure~\ref{q8T2}.
\begin{figure}
\centering
 \begin{tikzpicture}
[scale=1,auto=left,every node/.style={circle,scale=0.7}]

\path( 0.5+6.5,0)  node[shape=circle,draw,minimum size=0.8cm,label=above:,inner sep=0] (u1)  {$0$}
     ( 2.5+6.5,0)  node[shape=circle,draw,minimum size=0.8cm,label=above:,inner sep=0] (u2)  {$0$}
     ( 1.5+6.5,1.5)node[shape=circle,draw,minimum size=0.8cm,label=above:,inner sep=0] (u12) {$1$}

     ( 3.5+6.5,0)  node[shape=circle,draw,minimum size=0.8cm,label=above:,inner sep=0] (u3)  {$0$}
     ( 5.5+6.5,0)  node[shape=circle,draw,minimum size=0.8cm,label=above:,inner sep=0] (u4)  {$0$}
     ( 4.5+6.5,1.5)node[shape=circle,draw,minimum size=0.8cm,label=above:,inner sep=0] (u11) {$1$}

     ( 3+6.5,3.5)  node[shape=rectangle,draw,minimum size=0.8cm,label=above:,inner sep=0] (u13) {$-1$}

     (1.5+6.5,3.5) node[shape=circle,draw,minimum size=0.8cm,label=above:,inner sep=0] (u17) {$-1$}
     (1.5+6.5,4.5) node[shape=circle,draw,minimum size=0.8cm,label=above:,inner sep=0] (u18) {$-1$};

\draw[dotted](u17) -- (u18) ; 
\draw[-,below=2] (u13) edge node{$\frac{1}{\sqrt{t_0}}$} (u17) ; 
\draw[-,above] (u13) edge node{$\frac{1}{\sqrt{t_0}}$} (u18) ; 

\draw[decoration={brace,raise=5pt},decorate]
  (1.3+6.5,3.5) -- node[right=0pt,left,xshift=-0.5em] {$t_0$} (1.3+6.5,4.5);
  
\draw[dotted](u1) -- (u2) ; 
\draw[-](u1) edge node{$\frac{\sqrt{2}}{\sqrt{t_1}}$} (u12) ; 
\draw[-,pos=0.375,above](u2) edge node{$\frac{\sqrt{2}}{\sqrt{t_1}}$} (u12) ;

\draw[decoration={brace,mirror,raise=5pt},decorate]
  (0.5+6.5,-0.2) -- node[right=0pt,below,yshift=-0.5em] {$t_1$} (2.5+6.5,-0.2);
  
\draw[dotted](u3) -- (u4) ; 
\draw[-](u3) edge node{$\frac{\sqrt{2}}{\sqrt{t_p}}$}  (u11) ; 
\draw[-,pos=0.375,above](u4) edge node{$\frac{\sqrt{2}}{\sqrt{t_p}}$}  (u11) ; 

\draw[decoration={brace,mirror,raise=5pt},decorate]
  (3.5+6.5,-0.2) -- node[right=0pt,below,yshift=-0.5em] {$t_p$} (5.5+6.5,-0.2);

\draw[dotted](u12) -- (u11) ; 
\draw[-,below](u13) edge node{$\frac{\sqrt{5}}{\sqrt{p}}$}  (u12) ; 
\draw[-,below](u13) edge node{$\frac{\sqrt{5}}{\sqrt{p}}$}  (u11) ; 

\end{tikzpicture}
            \caption{A general unfolding $T_2 = T_2(t_0,\ldots,t_p)$ of $S_2$, rooted at the vertex $v_2$, which is depicted as a square. Vertex and edge-weights denote the entries of the matrix associated with it.} 
            \label{q8T2}
\end{figure}
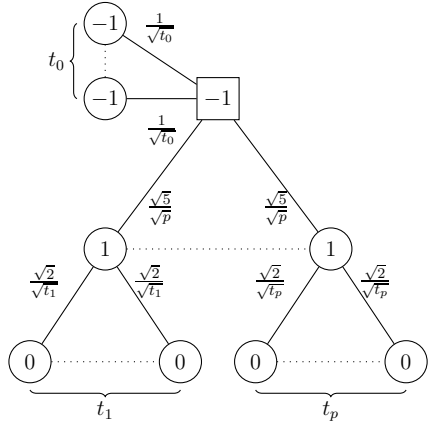
Lemma~\ref{spectraT2} gives information about the spectral of this matrix, which is always a subset of $\{-3,-1,0,2,3\}$.
The proof is in Appendix~\ref{appen}.
\begin{lemat}\label{spectraT2}
    If $M=M(T_2)$ be the matrix in Figure~\ref{q8T2} for the tree $T_2=T_2(t_0,\ldots,t_p)$, then $\DSpec(M)\subset\{-3,-1,0,2,3\}$. Moreover, $L(M,\lambda)=0$ for $\lambda\in\{-3,0,3\}$ and $m_{M[T_2\setminus v_2]}(\lambda) = m_M(\lambda)+1$ for $\lambda\in\{-1,2\}$. Also, \texttt{Diagonalize}$(M,-1)$ assigns $d_{v_2}=-4$ to the root $v_2$.
\end{lemat}

Next we consider unfoldings of $S_0$. Of course, possible unfolding depend on whether $S_0$ is a 2-path $P_2$ (if $S=S_7^{(8)}$), a 3-path $P_3$ (if $S=S_7^{(9)}$) or a 4-path $P_4$ (if $S=S_7^{(7)}$). If $S_0\simeq P_3$, all possible unfoldings of $S_0$ are isomorphic to a tree of type $T_1$, that is, there exist positive integers $p$, $s_1,\ldots,s_p$, for which the unfolding is given by $T_0=T_0(s_1,\ldots,s_p)$, where we first performed a $(p-1)$-CBD of the single branch of $S_0$ at $v$, and by performing $(s_i-1)$-CBDs of the leaves connected to each of the $p$ children of $v$. If $S_0\simeq P_4$, then possible unfoldings are isomorphic to  some $T_2$, and we obtain $T_0=T_0(s_0,\ldots,s_p)$ in the same way. 
If $S_0\simeq P_2$, then we just perform $(s_0-1)$-CBDs of the leaf connected to $v$ to produce $T_0=T_0(S_0)$. Figure~\ref{q8T0} depicts each of the three possibilities for $T_0$ and the weights of the matrix $M(T_0)$.
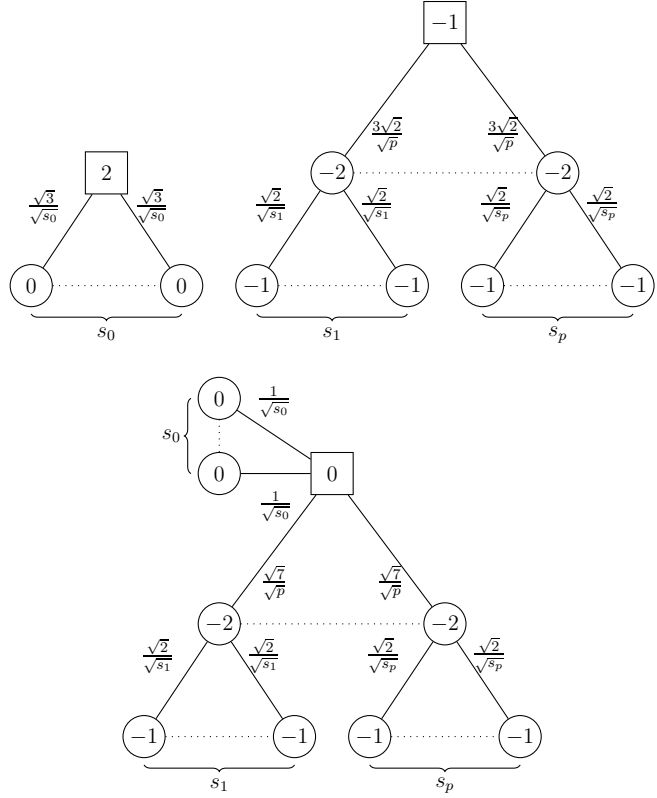
\begin{figure}
\centering
 \begin{tikzpicture}
[scale=1,auto=left,every node/.style={circle,scale=0.7}]

\path( 0.5-3,0)  node[shape=circle,draw,minimum size=0.8cm,label=above:,inner sep=0] (w1)  {$0$}
     ( 2.5-3,0)  node[shape=circle,draw,minimum size=0.8cm,label=above:,inner sep=0] (w2)  {$0$}
     ( 1.5-3,1.5)node[shape=rectangle,draw,minimum size=0.8cm,label=above:,inner sep=0] (w12) {$2$}
     
    ( 0.5,0)  node[shape=circle,draw,minimum size=0.8cm,label=above:,inner sep=0] (v1)  {$-1$}
     ( 2.5,0)  node[shape=circle,draw,minimum size=0.8cm,label=above:,inner sep=0] (v2)  {$-1$}
     ( 1.5,1.5)node[shape=circle,draw,minimum size=0.8cm,label=above:,inner sep=0] (v12) {$-2$}

     ( 3.5,0)  node[shape=circle,draw,minimum size=0.8cm,label=above:,inner sep=0] (v3)  {$-1$}
     ( 5.5,0)  node[shape=circle,draw,minimum size=0.8cm,label=above:,inner sep=0] (v4)  {$-1$}
     ( 4.5,1.5)node[shape=circle,draw,minimum size=0.8cm,label=above:,inner sep=0] (v11) {$-2$}

     ( 3,3.5)  node[shape=rectangle,draw,minimum size=0.8cm,label=above:,inner sep=0] (v13) {$-1$}

     ( 0.5+6-7.5,0-6)  node[shape=circle,draw,minimum size=0.8cm,label=above:,inner sep=0] (u1)  {$-1$}
     ( 2.5+6-7.5,0-6)  node[shape=circle,draw,minimum size=0.8cm,label=above:,inner sep=0] (u2)  {$-1$}
     ( 1.5+6-7.5,1.5-6)node[shape=circle,draw,minimum size=0.8cm,label=above:,inner sep=0] (u12) {$-2$}

     ( 3.5+6-7.5,0-6)  node[shape=circle,draw,minimum size=0.8cm,label=above:,inner sep=0] (u3)  {$-1$}
     ( 5.5+6-7.5,0-6)  node[shape=circle,draw,minimum size=0.8cm,label=above:,inner sep=0] (u4)  {$-1$}
     ( 4.5+6-7.5,1.5-6)node[shape=circle,draw,minimum size=0.8cm,label=above:,inner sep=0] (u11) {$-2$}

     ( 3+6-7.5,3.5-6)  node[shape=rectangle,draw,minimum size=0.8cm,label=above:,inner sep=0] (u13) {$0$}

     (1.5+6-7.5,3.5-6) node[shape=circle,draw,minimum size=0.8cm,label=above:,inner sep=0] (u17) {$0$}
     (1.5+6-7.5,4.5-6) node[shape=circle,draw,minimum size=0.8cm,label=above:,inner sep=0] (u18) {$0$};

\draw[dotted](w1) -- (w2) ; 
\draw[-](w1) edge node{$\frac{\sqrt{3}}{\sqrt{s_0}}$} (w12) ; 
\draw[-,pos=0.375,above](w2) edge node{$\frac{\sqrt{3}}{\sqrt{s_0}}$} (w12) ;
\draw[decoration={brace,mirror,raise=5pt},decorate]
  (0.5-3,-0.2) -- node[right=0pt,below,yshift=-0.5em] {$s_0$} (2.5-3,-0.2);

\draw[dotted](v1) -- (v2) ; 
\draw[-](v1) edge node{$\frac{\sqrt{2}}{\sqrt{s_1}}$} (v12) ; 
\draw[-,pos=0.375,above](v2) edge node{$\frac{\sqrt{2}}{\sqrt{s_1}}$} (v12) ;

\draw[decoration={brace,mirror,raise=5pt},decorate]
  (0.5,-0.2) -- node[right=0pt,below,yshift=-0.5em] {$s_1$} (2.5,-0.2);
  
\draw[dotted](v3) -- (v4) ; 
\draw[-](v3) edge node{$\frac{\sqrt{2}}{\sqrt{s_p}}$}  (v11) ; 
\draw[-,pos=0.375,above](v4) edge node{$\frac{\sqrt{2}}{\sqrt{s_p}}$}  (v11) ; 

\draw[decoration={brace,mirror,raise=5pt},decorate]
  (3.5,-0.2) -- node[right=0pt,below,yshift=-0.5em] {$s_p$} (5.5,-0.2);

\draw[dotted](v12) -- (v11) ; 
\draw[-,below](v13) edge node{$\frac{3\sqrt{2}}{\sqrt{p}}$}  (v12) ; 
\draw[-,below](v13) edge node{$\frac{3\sqrt{2}}{\sqrt{p}}$}  (v11) ;

\draw[dotted](u17) -- (u18) ; 
\draw[-,below=2] (u13) edge node{$\frac{1}{\sqrt{s_0}}$} (u17) ; 
\draw[-,above] (u13) edge node{$\frac{1}{\sqrt{s_0}}$} (u18) ; 

\draw[decoration={brace,raise=5pt},decorate]
  (1.3+6-7.5,3.5-6) -- node[right=0pt,left,xshift=-0.5em] {$s_0$} (1.3+6-7.5,4.5-6);
  
\draw[dotted](u1) -- (u2) ; 
\draw[-](u1) edge node{$\frac{\sqrt{2}}{\sqrt{s_1}}$} (u12) ; 
\draw[-,pos=0.375,above](u2) edge node{$\frac{\sqrt{2}}{\sqrt{s_1}}$} (u12) ;

\draw[decoration={brace,mirror,raise=5pt},decorate]
  (0.5+6-7.5,-0.2-6) -- node[right=0pt,below,yshift=-0.5em] {$s_1$} (2.5+6-7.5,-0.2-6);
  
\draw[dotted](u3) -- (u4) ; 
\draw[-](u3) edge node{$\frac{\sqrt{2}}{\sqrt{s_p}}$}  (u11) ; 
\draw[-,pos=0.375,above](u4) edge node{$\frac{\sqrt{2}}{\sqrt{s_p}}$}  (u11) ; 

\draw[decoration={brace,mirror,raise=5pt},decorate]
  (3.5+6-7.5,-0.2-6) -- node[right=0pt,below,yshift=-0.5em] {$s_p$} (5.5+6-7.5,-0.2-6);

\draw[dotted](u12) -- (u11) ; 
\draw[-,below](u13) edge node{$\frac{\sqrt{7}}{\sqrt{p}}$}  (u12) ; 
\draw[-,below](u13) edge node{$\frac{\sqrt{7}}{\sqrt{p}}$}  (u11) ;

\end{tikzpicture}
            \caption{A general unfolding $T_0(S_0)$ of $S_0\simeq P_2$ (top left), an unfolding $T_0(s_1,\ldots,s_p)$ of $S_0\simeq P_3$ (top right) and an unfolding $T_0(s_0,\ldots,s_p)$ of $S_0\simeq P_4$ (bottom), rooted at vertex $v$, which is depicted as a square.}
            \label{q8T0}
\end{figure}
Lemma~\ref{upperseed} gives information about the spectrum of the matrix $M(T_0)$ depending on $S_0$. Again, we refer to Appendix~\ref{appen} to see the proof of this lemma. In the statement of Lemma~\ref{upperseed} the notation $d_v(M,\lambda)$ stands for the assignment of \texttt{Diagonalize}$(M,-\lambda)$ to the vertex $v$.
\begin{lemat}\label{upperseed}
    Let $M=M(T_0)$ be the matrix in Figure~\ref{q8T0} for the tree $T_0=T_0(s_0)$ if $S_0\simeq P_2$, $T_0=T_0(s_1,\ldots,s_p)$  if $S_0\simeq P_3$ and $T_0=T_0(s_0,\ldots,s_p)$ if $S_0\simeq P_4$. The following hold:
    \begin{enumerate}
        \item[(i)] If $S_0\simeq P_2$, $\DSpec(M)\subset\{-1,0,3\}$, $L(M,\lambda)=0$ for $\lambda\in\{-1,3\}$, and $m_{M[T_0\setminus v]}(0) = m_M(0)+1$. Moreover, $d_v(M,-3)=4$, $d_v(M,1)=4$ and $d_v(M,2)=\frac{3}{2}$.
        \item[(ii)] If $S_0\simeq P_3$, $\DSpec(M)\subset\{-6,-3,-1,0,3\}$, $L(M,\lambda)=0$ for $\lambda\in\{-6,-1,3\}$, and $m_{[T_0\setminus v]}(\lambda) = m_M(\lambda)+1$ for $\lambda\in\{-3,0\}$. Moreover, $d_v(M,1)=\frac{7}{2}$ and $d_v(M,2)=\frac{3}{5}$.
        \item[(iii)] If $S_0\simeq P_4$, $\DSpec(M)\subset\{-5,-3,-1,0,3\}$, $L(M,\lambda)=0$ for $\lambda\in\{-5,-1,2\}$, and $m_{[T_0\setminus v]}(\lambda) = m_M(\lambda)+1$ for $\lambda\in\{-3,0\}$. Moreover, $d_v(M,1)=7$ and $d_v(M,2)=\frac{12}{5}$.
    \end{enumerate}
    \end{lemat}
Now we are ready to define the matrix $M\in\mathcal{S}(T)$ with the property that $q(T)\leq 8$. Given $T\in\mathcal{T}(S)$, consider the matrix $M(T_0)$ defined for $T_0$ in Figure~\ref{q8T0} and the matrices $M(T^{(i)}_1)$ and $M(T^{(j)}_2)$ defined for $T^{(i)}_1$ and $T^{(j)}_2$ in Figure~\ref{q8T1} and Figure~\ref{q8T2}, respectively. The full matrix $M=(m_{ij})\in\mathcal{S}(T)$ is defined as follows:
\begin{enumerate}
    \item[(i)] $M[T_0]=M(T_0)$, $M[T^{(i)}_1]=M(T^{(i)}_1)$ and $M[T^{(j)}_2]=M(T^{(j)}_2)$;
    \item[(ii)] The entries on edges between the root $v$ of $T_0$ and the roots $v^{(i)}_1$ of $T^{(i)}_1$ are given by$${m_{vv^{(i)}_1} =
\begin{cases}
\sqrt{\frac{5}{q_1}}, & \textrm{ if } S_0\simeq P_2,\\
\frac{2\sqrt{2}}{\sqrt{q_1}}, & \textrm{ if } S_0\simeq P_3,\\
\sqrt{\frac{7}{q_1}}, & \textrm{ if } S_0\simeq P_4.
\end{cases}}$$
\item[(iii)] The entries $m_{vv^{(j)}_2}$ on edges between the root $v$ of $T_0
$ and the roots $v^{(j)}_2$ of $T^{(j)}_2$ may be any non-zero real number.
\end{enumerate}

This matrix $M$ is depicted in Figure~\ref{matrixM}.
\begin{figure}
$$
\begin{pNiceMatrix}[first-row][first-col]
                       & V(U)       & V(T^{(1)}_1) & \cdots &V(T^{(q_1)}_1) & V(T^{(1)}_2) & \cdots  & V(T^{(q_2)}_2)   \\[0.2cm]
        V(U)           & M_0        & A_{11}     & \cdots & A_{1q_1}    & A_{21}     & \cdots  & A_{2q_2}       \\[0.3cm]
        V(T^{(1)}_1)   & A_{11}   & M^{(1)}_1    &\cdots  & \bf{0}        & \bf{0}       & \cdots  & \bf{0}           \\[0.3cm]
        \vdots         & \vdots     & \vdots       & \ddots & \vdots        & \vdots       & \vdots  & \vdots         \\[0.2cm]
        V(T^{(q_1)}_1) & A_{1q_1} & \bf{0}       & \cdots & M^{(q_1)}_1   & \bf{0}       & \dots  & \bf{0}           \\[0.2cm]
        V(T^{(1)}_2)   & A_{21}   & \bf{0}       & \cdots & \bf{0}        & M^{(1)}_2    & \cdots  & \bf{0}           \\[0.2cm]
        \vdots~~       & \vdots     & \vdots       & \vdots & \vdots        & \vdots       & \ddots  & \vdots          \\[0.2cm]
        V(T^{(q_2)}_2) & A_{2q_2} & \bf{0}       & \cdots & \bf{0}        & \bf{0}       & \dots   & M^{(q_2)}_2      \\
    \end{pNiceMatrix} \text{, where }$$
$$A_{1i} = \begin{pNiceMatrix}[first-row][first-col]
            & v^{(i)}_1    &        &        &     \\
        v & m_{vv^{(i)}_1} & 0      & \cdots & 0   \\
            &   0    & 0      & \cdots & 0   \\
            & \vdots & \vdots & \ddots & 0   \\
            &   0    & 0      & 0      & 0   \\
    \end{pNiceMatrix}
    \text{ and }
    A_{2j} = \begin{pNiceMatrix}[first-row][first-col]
            & v^{(j)}_2    &        &        &     \\
        v & m_{vv^{(j)}_2} & 0      & \cdots & 0   \\
            &   0    & 0      & \cdots & 0   \\
            & \vdots & \vdots & \ddots & 0   \\
            &   0    & 0      & 0      & 0   \\
    \end{pNiceMatrix}
$$
\caption{Matrix $M$.}
\label{matrixM}
\end{figure}

Theorem~\ref{q8thm} gives the distinct eigenvalues of $M$. 
\begin{theorem}\label{q8thm}
    Let $T$ be an unfolding of $S^{(7)}_7$, $S^{(8)}_7$ or $S^{(9)}_7$. Let $M\in\mathcal{S}(T)$ be the matrix in the Figure~\ref{matrixM}. Then, $\DSpec(M)\subset\{\lambda_{\min},-3,-1,0,1,2,3,\lambda_{\max}\}$. In particular, $q(T)\leq8$. 
\end{theorem}

\begin{proof}
    We write the proof for $S_0\simeq P_2$, it is analogous for $S_0\simeq P_3$ and for $S_0\simeq P_4$. We first state some additional notation, namely $m_0(\lambda)=m_{M(T_0)}(\lambda)$, $m_1(\lambda)=\sum_{i=1}^{q_1}m_{M(T^{(i)}_1)}(\lambda)$ and $m_2(\lambda)=\sum_{j=1}^{q_2}m_{M(T^{(j)}_2)}(\lambda)$ for $\lambda\in\Lambda=\{-3,-1,0,1,2,3\}$. 
    
Since the sum of the multiplicities of the eigenvalues of a matrix is equal to its dimension
and since the spectra of the different submatrices of $M$ are described in Lemmas~\ref{spectraT1},~\ref{spectraT2}, and~\ref{upperseed}, we have, 
    \begin{align*}
        n_0 &= m_{0}(-1)+m_{0}(0)+m_{0}(3),\\
        n_1 &= m_{1}(-3)+m_{1}(-1)+m_{1}(0)+m_{1}(1)+m_{1}(3),\\
        n_2 &= m_{2}(-3)+m_{2}(-1)+m_{2}(0)+m_{2}(2)+m_{2}(3),
    \end{align*}
    where $n_0=|V(T_0)|$, $n_1=\sum_{i=1}^{q_1}|V(T^{(i)}_1)|$ and $n_2=\sum_{j=1}^{q_2}|V(T^{(j)}_2)|$.
    
    Then, we have
    \begin{align}\label{eqmulti}
n=n_0+n_1+n_2&=m_0(-1)+m_0(0)+m_0(3)+\nonumber\\
    &+m_{1}(0)+m_{1}(-1)+m_{1}(1)+m_{1}(-3)+m_{1}(3)+\nonumber\\&+m_{2}(0)+m_{2}(-1)+m_{2}(2)+m_{2}(-3)+m_{2}(3)\nonumber\\
    &\stackrel{(*)}{=}2+m_0(0)+\nonumber\\
    &+m_{1}(0)+m_{1}(-1)+m_{1}(1)+m_{1}(-3)+m_{1}(3)+\nonumber\\
    &+m_{2}(0)+m_{2}(-1)+m_{2}(2)+m_{2}(-3)+m_{2}(3).
    \end{align}
    Equation (*) holds owing to the fact that any $\lambda$ such that $L(M(T_0),\lambda)=0$ must have multiplicity 1.

    The claim below describes how $m_M(\lambda)$ relates with $m_0(\lambda)$,$m_1(\lambda)$, and $m_2(\lambda)$ when $S_0\simeq P_2$. 
    \begin{claim}\label{multiplicityU0}
    If $S_0\simeq P_2$ the following hold:
    \begin{enumerate}
        \item [(i)] $m_M(0) = m_{0}(0)+m_{1}(0)+m_{2}(0)$,
        \item [(ii)] $m_M(-1) = m_{1}(-1)+m_{2}(-1)+1$,
        \item [(iii)] $m_M(1) = m_{1}(1)$,
        \item [(iv)] $m_M(2) = m_{2}(2)+1$,
        \item [(v)] $m_M(-3) = m_{1}(-3)+m_{2}(-3)-1$,
        \item [(vi)] $m_M(3) = m_{1}(3)+m_{2}(3)-1$.     
    \end{enumerate}
    \end{claim}

    \begin{proof}


For each $\lambda\in\{-3,-1,0,1,2,3\}$, we know that $m_M(\lambda)$ is the number of zeros in an application of \texttt{Diagonalize}$(M,-\lambda)$. Similarly, $m_0(\lambda)$ is the number of zeros that appear in \texttt{Diagonalize}$(M[T_0],-\lambda)$, while 
$m_1(\lambda)$ is the sum of the number of zeros that appear in \texttt{Diagonalize}$(M[T_1^{(i)}],-\lambda)$ (for all $i$), and $m_2(\lambda)$ is the sum of the number of zeros that appear in \texttt{Diagonalize}$(M[T_2^{(j)}],-\lambda)$ (for all $j$). 
The only vertices $w$ for which the value of $d_w$ may be different when algorithm \texttt{Diagonalize}$(M,-\lambda)$ is applied to the entire tree rather than its branches are:
\begin{itemize}
\item[(a)] $w=v$. For the value to be different, one possibility is that no child of $v$ in $T_0$ is assigned value 0 by \texttt{Diagonalize}$(M[T_0],-\lambda)$ and at least one of the roots $v_1^{(i)}$ or $v_2^{(j)}$ is assigned value 0 in the corresponding branch. The other possibility is that none of these values are 0 and 
\begin{equation}\label{main:eq}d_v(T)=d_v(T_0)-\sum_{i=1}^{q_1} \frac{\delta_{v_1^{(i)}}
}{d_{v_1^{(i)}}} -\sum_{j=2}^{q_2} \frac{\delta_{v_2^{(j)}
}}{d_{v_2^{(j)}}},
\end{equation}
where $\delta_{v_1^{(i)}}$ is equal to 0 if $v_1^{(i)}$ has a child that was assigned value 0, and $\delta_{v_1^{(i)}}=(m_{vv_1^{(i)}})^2$ otherwise. Similarly,
$\delta_{v_2^{(j)}}=0$ if $v_2^{(j)}$ has a child that was assigned value 0, and
$\delta_{v_2^{(j)}}=(m_{vv_2^{(j)}})^2$ otherwise.

\item[(b)] $w=v_1^{(i)}$, \texttt{Diagonalize}$(M[T_1^{(i)}],-\lambda)$ assigns 0 to $v_1^{(i)}$, and this value is replaced by 2 when processing $v$.

\item[(c)] $w=v_2^{(j)}$, \texttt{Diagonalize}$(M[T_2^{(j)}],-\lambda)$ assigns 0 to $v_2^{(j)}$, and this value is replaced by 2 when processing $v$.

\end{itemize}

To illustrate this approach, we consider the cases $\lambda=0$, $\lambda=2$ and $\lambda=-3$. For all other values of $\lambda$, we may apply the same strategy. 

First consider $\lambda=0$. By Lemma~\ref{spectraT1}, it is an eigenvalue of $T_1^{(i)}$ for every $i$, and $d_{v^{(i)}_1}=0$. The same happens for $T_2^{(j)}$ for every $j$ by Lemma~\ref{spectraT2}. Finally, by the definition of $T_0$, each child of $v$ is initialized with 0 when applying \texttt{Diagonalize}$(M[T_0],0)$. As a consequence, when we process $v$ in \texttt{Diagonalize}$(M,0)$ one of its children is assigned value $2$, and it will be assigned a negative value. We may suppose that the child lies in $T_0$, which is precisely what would happen for $v$ in \texttt{Diagonalize}$(M[T_0],0)$. As a consequence
$$m_M(0)=m_{0}(0)+m_{1}(0)+m_{2}(0).$$

Consider $\lambda=2$. By Lemma~\ref{spectraT1}, it is never an eigenvalue of $T_1^{(i)}$. Moreover, all roots $v_1^{(i)}$ are assigned $d_{v_1^{(i)}}=10/3$. By Lemma~\ref{spectraT2}, we know that $m_{M[T^{(j)}_2\setminus v^{(j)}_2]}(2)=m_{M[T^{(j)}_2]}(2)+1$ for every $j$. This means that every root $
v^{(j)}_2$ had a child that was assigned value 0 by \texttt{Diagonalize}$(M[T_2^{(j)}],-2)$, so that the same happens for \texttt{Diagonalize}$(M,-2)$. In particular, when processing $v^{(j)}_2$, the algorithm deletes the edge to its parent $v$. Finally, by Lemma~\ref{upperseed}, we have that $\lambda=2$ is not an eigenvalue of $T_0$ and $d_v(T_0) = 3/2$. Equation~\eqref{main:eq} leads to
$$d_v(T) =\frac{3}{2}-\sum_{\ell=1}^{q_1}\frac{5/q_1}{10/3}=\frac{3}{2}-\frac{15}{10}=0.$$
This implies that $m_M(2)=m_{2}(2)+1.$

Next, consider $\lambda=-3$. Lemmas~\ref{spectraT1} and~\ref{spectraT2} tell us that that it is an eigenvalue of all $T_1^{(i)}$ and $T_2^{(j)}$, and the algorithm assigns 0 to all $v_1^{(i)}$ and $v_2^{(j)}$. The children of $v$ in $T_0$ are not assigned value 0, and by Lemma~\ref{upperseed}, $v$ is not assigned value 0 in \texttt{Diagonalize}$(M[T_0],3)$ When processing $v$ in \texttt{Diagonalize}$(M,3)$, one of the $v_1^{(i)}$ or $v_2^{(j)}$ is assigned value 2, and $d_v$ is assigned a negative number. So $m_M(-3)=m_{1}(-3)+m_{2}(-3)-1.$

    \end{proof}

Coming back to the proof of Theorem~\ref{q8thm}, Claim~\ref{multiplicityU0} leads to 
\begin{eqnarray*}
\sum_{\lambda\in\Lambda}m_M(\lambda)&=&m_0(0)+\left(m_1(0)+m_1(-1)+m_1(1)+m_1(-3)+m_1(3)\right)\\&&+\left(m_2(0)+m_2(-1)+m_2(2)+m_2(-3)+m_2(3)\right)\\
&\stackrel{\eqref{eqmulti}}{=}&n-2.
\end{eqnarray*}
So, in addition the eigenvalues in $\Lambda$ we can have at most two other eigenvalues. 
Moreover, by Theorem~\ref{thm:simpleroots}, we know that $-3$ and $3$ cannot be $\lambda_{\min}$ and $\lambda_{\max}$, since \texttt{Diagonalize}$(M,-3)$ and \texttt{Diagonalize}$(M,3)$ assign 0 to vertices other than the root $v$. As a consequence, the missing eigenvalues must be $\lambda_{\min}$ and $\lambda_{\max}$, proving our result.
\end{proof}


\section*{Acknowledgments}
The authors thank CAPES for the partial support under project MATH-AMSUD 88881.694479/2022-01. They also acknowledge partial support by CNPq (Proj.\ 408180/2023-4). L.\ E.\ Allem was partially supported by FAPERGS 21/2551-
0002053-9. C.~Hoppen was partially supported of CNPq (Proj.\ 315132/2021-3). CAPES is Coordena\c{c}\~{a}o de Aperfei\c{c}oamento de Pessoal de N\'{i}vel Superior. CNPq is Conselho Nacional de Desenvolvimento Cient\'{i}fico e Tecnol\'{o}gico. FAPERGS is Funda\c{c}\~{a}o de Amparo \`{a} Pesquisa do Estado do Rio Grande do Sul.

\bibliographystyle{amsplain}

\providecommand{\bysame}{\leavevmode\hbox to3em{\hrulefill}\thinspace}
\providecommand{\MR}{\relax\ifhmode\unskip\space\fi MR }
\providecommand{\MRhref}[2]{%
  \href{http://www.ams.org/mathscinet-getitem?mr=#1}{#2}
}
\providecommand{\href}[2]{#2}


\begin{thebibliography}{10}

\bibitem{ahmadi2013}
Bahman Ahmadi, Fatemeh Alinaghipour, Michael Cavers, Shaun Fallat, Karen Meagher, and Shahla Nasserasr, \emph{Minimum number of distinct eigenvalues of graphs}, ELA. The Electronic Journal of Linear Algebra [electronic only] \textbf{26} (2013).

\bibitem{allem2023diminimal}
L~Emilio Allem, Rodrigo~O Braga, Carlos Hoppen, Elismar~R Oliveira, Lucas~S Sibemberg, and Vilmar Trevisan, \emph{Diminimal families of arbitrary diameter}, Linear Algebra and its Applications \textbf{676} (2023), 318--351.

\bibitem{barioli2004two}
Francesco Barioli and Shaun Fallat, \emph{On two conjectures regarding an inverse eigenvalue problem for acyclic symmetric matrices}, The Electronic Journal of Linear Algebra \textbf{11} (2004), 41--50.

\bibitem{barrett2020}
Wayne Barrett, Steve Butler, Shaun~M. Fallat, H.~Tracy Hall, Leslie Hogben, Jephian C.-H. Lin, Bryan~L. Shader, and Michael Young, \emph{The inverse eigenvalue problem of a graph: Multiplicities and minors}, Journal of Combinatorial Theory, Series B \textbf{142} (2020), 276--306.

\bibitem{TEMA1041}
Rodrigo Braga and Virg\'{i}nia Rodrigues, \emph{Locating eigenvalues of perturbed {L}aplacian matrices of trees}, TEMA (S\~ao Carlos) Brazilian Soc. of Appl. Math. and Comp. \textbf{18} (2017), no.~3, 479--491.

\bibitem{fallat2022}
Shaun Fallat and Seyed~Ahmad Mojallal, \emph{On the minimum number of distinct eigenvalues of a threshold graph}, Linear Algebra and its Applications \textbf{642} (2022), 1--29.

\bibitem{diagonalize}
Carlos Hoppen, David~P Jacobs, and Vilmar Trevisan, \emph{Locating eigenvalues in graphs: Algorithms and applications}, Springer Nature, 2022.

\bibitem{JT2011}
David~P. Jacobs and Vilmar Trevisan, \emph{Locating the eigenvalues of trees}, Linear Algebra Appl. \textbf{434} (2011), no.~1, 81--88. \MR{2737233 (2012b:15017)}

\bibitem{Proyecciones2018}
Charles~R. Johnson, Jacob Lettie, Sander Mack-Crane, and Alicja Szabelska-Bersewicz, \emph{{Branch duplication in trees: uniqueness of seeds and enumeration of seeds}}, {Proyecciones (Antofagasta)} \textbf{39} (2020), 451 -- 465 (en).

\bibitem{JSdiminimal}
Charles~R. Johnson and Carlos~M. Saiago, \emph{Diameter minimal trees}, Linear and Multilinear Algebra \textbf{64} (2016), no.~3, 557--571.

\bibitem{JohnsonSaiago2018}
\bysame, \emph{Eigenvalues, multiplicities and graphs}, Cambridge University Press, Cambridge, UK, 2018.

\bibitem{leal2002minimum}
Ant{\'o}nio Leal-Duarte and Charles~R Johnson, \emph{On the minimum number of distinct eigenvalues for a symmetric matrix whose graph is a given tree}, Mathematical Inequalities and Applications \textbf{5} (2002), 175--180.

\end{thebibliography}

\appendix

\section{Proof of Lemmas~\ref{spectraT1},~\ref{spectraT2} and \ref{upperseed}}\label{appen}

\begin{proof}[Proof of Lemma~\ref{spectraT1}]

Let $M$ be the matrix in Figure~\ref{q8T1} for the tree $T_1 = T_1(t_1,\ldots,t_p)$.

We apply \texttt{Diagonalize}$(M,0)$. Every leaf $w$ is assigned $d_w = 0$. So, we process the parents $u_1,\ldots,u_p$ of the leaves. Since each $u_i$ has a child $w_i$ with $d_{w_i} = 0$, $d_{u_i}$ is assigned a negative value and $d_{w_i}=2$. The edge between $u_i$ and its parent $v_1$ is removed. Then, the initial value $d_{v_1}=0$ remains unchanged. In the end, the number of zeros in the final assignment is $1+\sum_{i=1}^{p}(t_i-1)=|V(T_1)|-2p$. By Theorem~\ref{thm_localizacao}(c) the multiplicity of $0$ is $|V(T_1)|-2p$ with $L(M,0)=0$.

Next we apply \texttt{Diagonalize}$(M,1)$. Each leaf $w$ receives $d_w=1$. We process the parents $u_1,\ldots,u_p$ of the leaves. For each $i$, $d_{u_i}=1-\sum_{i=1}^{t_i}\frac{1/t_i}{1}=1-1=0$. Now,  $v_1$ has a child $u$ with $d_u=0$, so when processing $v_1$, $d_{v_1}$ receives a negative assignment and $d_u=2$. So, in the end, we have $p-1$ zeros. By Theorem~\ref{thm_localizacao}(c) the multiplicity of $-1$ is $p-1$. Of course, in $M[T_1\setminus v_1]$, we would have exactly $p$ zero assignments, so $m_{M[T_1\setminus v_1]}(-1) = m_M(-1)+1$.

Consider \texttt{Diagonalize}$(M,-1)$. Each leaf $w$ receives $d_w=-1$. We process the parents $u_1,\ldots,u_p$ of the leaves. For each $i$, $d_{u_i}=-1-\sum_{i=1}^{t_i}\frac{1/t_i}{-1}=-1+1=0$. Now,  $v_1$ has a child $u$ with $d_u=0$, so when processing $v_1$, $d_{v_1}$ receives a negative assignment and $d_u=2$. So, in the end, we have $p-1$ zeros. By Theorem~\ref{thm_localizacao}(c) the multiplicity of $1$ is $p-1$. Of course, in $M[T_1\setminus v_1]$ we would have exactly $p$ zero assignments, so $m_{M[T_1\setminus v_1]}(1) = m_M(1)+1$.

Now we apply \texttt{Diagonalize}$(M,3)$. Each leaf $w$ receives $d_w=3$. We process the parents $u_1,\ldots,u_p$ of the leaves. For each $i$, $d_{u_i}=3-\sum_{i=1}^{t_i}\frac{1/t_i}{3}=3-\frac{1}{3}=\frac{8}{3}$. Now, when processing $v_1$ we get
$$d_{v_1} = 3 - \sum_{i=1}^p\frac{8/p}{8/3}=0.$$
So $d_{v_1}=0$, $m_M(-3)=1$ and $L(M,-3)=0$.

Using the trace of the matrix, we know that
$$0 = -3 + (-1)(p-1) + 0(|V(T_1)|-2p)+ (+1)(p-1) +\lambda_{\text{max}},$$
so we have that $\lambda_{\text{max}}=3$. And, by Lemma~\ref{thm:simpleroots} we know that $L(M,3)=0$.

Finally, when performing \texttt{Diagonalize}$(M,-2)$, the leaves are assigned $-2$. The parent $z$ of each a set of $t_z$ leaves, $z\in\{1,\ldots,p\}$, gets
        $$-2-\sum_{\ell=1}^{t_z}\frac{1/t_z}{-2}=-2+\frac{1}{2}=-\frac{3}{2}.$$
        The root $v_1$ is assigned
        $$d_{v_1}=-2-\sum_{\ell=1}^{p_i}\frac{8/p_i}{-3/2}=-2+\frac{16}{3}=\frac{10}{3}.$$

\end{proof}

\begin{proof}[Proof of Lemma~\ref{spectraT2}]

Let $M$ be the matrix on the Figure~\ref{q8T2} for the tree $T_2 = T_2(t_0,\ldots,t_p)$. Recall that the root is $v_2$ and its non-leaf children are $u_1,\ldots,u_p$.

We apply \texttt{Diagonalize}$(M,0)$. Every leaf $w$ connected to $u_1,\ldots,u_p$ is assigned $d_w = 0$. Every leaf $w'$ connected to $v_2$ is assigned $d_{w'} = -1$. Since each $u_i$ has a child $w_i$ with $d_{w_i} = 0$, $d_{u_i}$ is assigned a negative value, $d_{w_i}=2$, and the edge between $u_i$ and $v_2$ is removed. 
In the final step, $v_2$ get the following assignment
$$d_{v_2}=-1-\sum_{i=1}^{t_0}\frac{1/t_0}{-1}=-1+1=0.$$
In the end, we have $1+\sum_{i=1}^{p}(t_i-1)=|V(T_2)|-2p-t_0$ zero assignments. By Theorem~\ref{thm_localizacao}(c) the multiplicity of $0$ is $|T_1|-2p$ with $L(M,0)=0$.

Next we apply \texttt{Diagonalize}$(M,1)$. Each leaf $w$ connected to $u_1,\ldots,u_p$ is assigned $d_w=1$. Every leaf $w'$ connected to $v_2$ will is assigned $d_{w'} = 0$. For each $i$, $d_{u_i}=1-\sum_{\ell=1}^{t_i}\frac{1/t_i}{1}=1-1=0$. Before processing $v_2$, all of its children have $d_u=0$. So, upon processing $v_2$, one of the children is assigned $2$ and $v_2$ is assigned a negative value.
In the end, we have $p-1+t_0$ zero assignments. So, by Theorem~\ref{thm_localizacao}(c), $m_M(-1)=p-1+t_0$. Of course, for $M[T_2\setminus v_2]$ we have exactly $p+t_0$ zero assignments, so $m_{[T_2\setminus v_2]}(-1) = m_M(-1)+1$.

Now we apply \texttt{Diagonalize}$(M,-2)$. Each leaf $w$ connected to $u_1,\ldots,u_p$ is assigned $d_w=-2$. Every leaf $w'$ connected to $v_2$ is assigned $d_{w'} = -3$. For each $i$, $d_{u_i}=-1-\sum_{\ell=1}^{t_i}\frac{2/t_i}{-2}=-1+1=0$.
Now, $v_2$ has a child $u$ with $d_u=0$, so when processing $v_2$, $d_{v_2}$ becomes negative and $d_u=2$. So, in the end, we have $p-1$ zero assignments. So, by Theorem~\ref{thm_localizacao}(c), $m_M(2)=p-1$. Of course, in $M[T_2\setminus v_2]$ we would have exactly $p$ zero assignments, so $m_{M[T_2\setminus v_2]}(1) = m_M(1)+1$.

Now we apply \texttt{Diagonalize}$(M,3)$. Each leaf $w$ connected to $u_1,\ldots,u_p$ is assigned $d_w=3$. Every leaf $w'$ connected to $v_2$ is assigned $d_{w'} = 2$. 
For each $i$, $d_{u_i}=1+3-\sum_{\ell=1}^{t_i}\frac{2/t_i}{3}=4-\frac{2}{3}=\frac{10}{3}$. Now, when processing $v_2$ we will have
$$d_{v_2} = -1+3 - \sum_{i=1}^p\frac{5/p}{10/3}- \sum_{i=1}^{t_0}\frac{1/t_0}{2}=2-\frac{3}{2}-\frac{1}{2}=0.$$
So we have one final $d_{v_2}=0$ assignment, so $m_M(-3)=1$ and $L(M,-3)=0$.

Using the trace of the matrix, we know that
$$p-(t_0+1) = -3 + (-1)(p-1+t_0) + 0(|V(T_2)|-2p-t_0)+ 2(p-1) +\lambda_{\text{max}},$$
which leads to
$$3 = -p+(t_0+1)-(p-1)-t_0 + (p-1) + (p-1) +\lambda_{\text{max}},$$
so we have that $\lambda_{\text{max}}=3$. And, by Lemma~\ref{thm:simpleroots} we know that $m_M(3)=1$ and $L(M,3)=0$.

Finally, when performing \texttt{Diagonalize}$(M,-1)$, the leaves connected to $u_1,\ldots,u_p$ are assigned $-1$ and the leaves connected to $v_2$ are assigned -2. Each $u_i$ is assigned
        $$0-\sum_{\ell=1}^{t_i}\frac{2/t_i}{-1}=2.$$
        The root $v_2$ is assigned
        $$d_{v_2}=-2-\sum_{\ell=1}^{t_0}\frac{1/t_0}{-2}-\sum_{\ell=1}^{p}\frac{5/p}{2}=-2+\frac{1}{2}-\frac{5}{2}=-4.$$

\end{proof}

\begin{proof}[Proof of Lemma~\ref{upperseed}]

We focus on the case $S_0\simeq P_2$. If $S_0\simeq P_3$ we refer the reader to the proof of Lemma~\ref{spectraT1}, which deals with an isomorphic tree with different weights. 
For $S_0\simeq P_4$ we refer the reader to the proof of Lemma~\ref{spectraT2}, which also deals with an isomorphic tree with different weights. 

For the proof for $S_0\simeq P_2$, let $M$ be the matrix on the Figure~\ref{upperseed} (on the left) for the tree $T_0 = T_0(s_0)$.

We apply \texttt{Diagonalize}$(M,0)$. Every child $w$ of $v$ is initially assigned $d_w = 0$. The value of one of them is replaced by 2 when processing $v$, and $v$ is assigned a negative value. By Theorem~\ref{thm_localizacao}(c) the multiplicity of $0$ is $s_0-1=|V(T_0)|-2$. Of course, in $M[T_0\setminus v]$ we would have exactly $s_0$ zero assignments, so $m_{M[T_0\setminus v]}(0) = m_M(0)+1$.

Next we apply \texttt{Diagonalize}$(M,1)$. Every child $w$ of $v$ is assigned $d_w=1$. When processing $v$ we get
$$d_{v} = 2 + 1 - \sum_{i=1}^{s_0}\frac{3/s_0}{1}=3-3=0.$$
So we have one final $d_{v}=0$ assignment, so $m_M(-1)=1$ and $L(M,-1)=0$.

Using the trace of the matrix, we know that
$$2+0 \cdot s_0 = -1 + 0\cdot (s_0-1) +\lambda_{\text{max}},$$
which leads to
$$\lambda_{\text{max}} =3,$$ 
and, by Lemma~\ref{thm:simpleroots} we know that $m_M(3)=1$ and $L(M,3)=0$.

Finally, \texttt{Diagonalize}$(M,3)$ assigns $d_w=3$ to each leaf $w$. Then, when processing $v$ we get
$$d_{v} = 5 - \sum_{i=1}^{s_0}\frac{3/s_0}{3}=5-1=4.$$
\texttt{Diagonalize}$(M,-1)$ assigns $d_w=-1$ to each leaf $w$. Then, when processing $v$ get 
$$d_{v} = 1 - \sum_{i=1}^{s_0}\frac{3/s_0}{-1}=1+3=4.$$
\texttt{Diagonalize}$(M,-2)$ assigns $d_w=-2$ for each leaf $w$. Then, when processing $v$ we get
$$d_{v} = 0 - \sum_{i=1}^{s_0}\frac{3/s_0}{-2}=\frac{3}{2}.$$
\end{proof}

\end{document}